\documentclass[11pt,letterpaper,reqno]{amsart}

\usepackage[T1]{fontenc}
\usepackage{amsmath,amssymb,amsthm,amsfonts}
\usepackage{mathrsfs}
\usepackage{booktabs}
\usepackage{enumitem}
\usepackage{tikz}
\usepackage{hyperref}
\usepackage[capitalize,noabbrev]{cleveref}

\hypersetup{
  pdfstartview={FitH},
  colorlinks=true,
  linkcolor=blue,
  citecolor=blue,
  urlcolor=blue,
  pdftitle={Bernstein Transfers and Greedy Records for Fence and Circular-Fence Order Polynomials},
  pdfauthor={Pyuyi Chufeng Huang},
  pdfsubject={Enumerative combinatorics and poset theory},
  pdfkeywords={order polynomial, fence poset, permutation statistic,
    right-to-left record, record set, Bernstein basis, explicit bijection,
    quasisymmetric function, circular fence}
}

\newtheorem{theorem}{Theorem}[section]
\newtheorem{lemma}[theorem]{Lemma}
\newtheorem{proposition}[theorem]{Proposition}
\newtheorem{corollary}[theorem]{Corollary}
\theoremstyle{definition}
\newtheorem{definition}[theorem]{Definition}
\newtheorem{example}[theorem]{Example}
\newtheorem{remark}[theorem]{Remark}

\DeclareMathOperator{\rec}{rec}
\DeclareMathOperator{\zzrec}{zzrec}
\newcommand{\dd}{\,\mathrm{d}}
\newcommand{\eps}{\varepsilon}

\title[Bernstein transfers and greedy records]
{Bernstein Transfers and Greedy Records for
Fence and Circular-Fence Order Polynomials}

\author{Pyuyi Chufeng Huang}
\address{School of Mathematics, Sichuan University, Chengdu 610064, China}
\email{pyuyi233@gmail.com}


\begin{document}

\begin{abstract}
Let \(P_\eps\) be the fence poset associated with an orientation \(\eps\in\{+,-\}^{n-1}\) of a path.  We define a greedy right-to-left record statistic \(\rec_\eps\) on \(S_n\) and prove
\[
\sum_{\pi\in S_n}t^{\rec_\eps(\pi)}=n!\Omega(P_\eps;t),
\]
by a Bernstein-basis transfer between a continuous threshold recurrence and endpoint-refined order-preserving maps.  Under reflection, this statistic agrees pointwise with Kahane's independently obtained greedy block statistic; the alternating specialization gives the zig-zag case posed by Ferroni, Morales, and Panova.  A finite form of the transfer yields a direct recursive bijection for \(m\le n\), extended algorithmically to arbitrary alphabets. Refining by record set, direction, and terminal value identifies fixed fibers with decorated endpoint paths and pointed linear extensions of posets whose cover graphs are caterpillars.  We also define cyclic records for every nonconstant orientation \(\eta\) of a cycle and prove
\[
\sum_{\pi\in S_n}t^{\operatorname{crec}_\eta(\pi)}
=n!\Omega(C_\eta;t).
\]
When the Hasse diagram is a cycle, reflection identifies these records with Kahane's circular blocks and establishes his circular-fence conjecture.
\end{abstract}

\maketitle

\section{Introduction}

Let \(P\) be a finite poset.  Its order polynomial \(\Omega(P;t)\) is
characterized by
\[
\Omega(P;m)=\#\{f:P\to[m]\text{ order-preserving}\}
\]
for all positive integers \(m\); see \cite[Chapter~3]{StanleyEC1}.  Order
polynomials are closely tied to order polytopes
\cite{StanleyTwoPosetPolytopes} and hence to Ehrhart theory.  Ferroni,
Morales, and Panova proved coefficient nonnegativity for skew-shape cell posets,
hence for fences, and separately for circular fences
\cite[Corollary~5.1 and Theorem~7.7]{FMP}.  For the alternating, or zig-zag,
fence \(Z_n\), they asked for a direct permutation-statistic interpretation of
\(n!\Omega(Z_n;t)\) \cite[Problem~5.3]{FMP}.

Independently, Kahane introduced a greedy block statistic and proved the
corresponding identity for every fence; he also obtained a lower bound for the
linear coefficient of order-polytope Ehrhart polynomials and a coefficient
interpretation for Ehrhart polynomials of Schubert-matroid base polytopes
\cite[Sections~5 and 6]{Kahane}.  After reflection,
\Cref{prop:kahane-path-comparison} identifies the statistic used here
pointwise with Kahane's.  The additional path results are the Bernstein
transfer, its finite and bijective forms, and the fixed-fiber refinements
developed below.  The cyclic construction is separate and gives a record
formula for every nonconstant orientation of a cycle.

We work with an arbitrary orientation
\(\eps=(\eps_1,\ldots,\eps_{n-1})\in\{+,-\}^{n-1}\) of a path.  Let
\(P_\eps\) be the transitive closure of
\[
\eps_i=+ \Longrightarrow x_i\succ x_{i+1},
\qquad
\eps_i=- \Longrightarrow x_i\prec x_{i+1}
\]
on \(\{x_1,\ldots,x_n\}\).  Thus \(Z_n=P_{\eps^{\mathrm{zz}}}\), where
\(\eps_i^{\mathrm{zz}}=+\) for odd \(i\) and \(\eps_i^{\mathrm{zz}}=-\) for
even \(i\).

For \(\pi=\pi_1\cdots\pi_n\in S_n\), the statistic starts with threshold
\(\pi_n\) and scans from right to left.  At a \(+\)-edge it records a value
larger than the current threshold, and at a \(-\)-edge a value smaller than
the threshold; every record resets the threshold.  The total number of
records is \(\rec_\eps(\pi)\).  A precise definition and an equivalent
greedy-chain formulation are given in
\Cref{def:records-chain,prop:chain-threshold-equivalence}.

\begin{theorem}\label{thm:main-general}
For every \(n\ge1\) and every \(\eps\in\{+,-\}^{n-1}\),
\[
\sum_{\pi\in S_n} t^{\rec_\eps(\pi)}
=
n!\,\Omega(P_\eps;t).
\]
\end{theorem}

For alternating signs, \Cref{thm:main-general} specializes to the zig-zag case
of \cite[Problem~5.3]{FMP}.  When all signs are \(+\) or all are \(-\), it
specializes to the classical right-to-left maximum or minimum statistic.

The relation with Kahane's statistic is pointwise: if
\[
\delta_j=-\eps_{n-j},
\qquad
\pi^{\mathrm{rev}}_j=\pi_{n+1-j},
\]
then \Cref{prop:kahane-path-comparison} gives
\[
\rec_\eps(\pi)
=\operatorname{bl}_{P_\delta}(\pi^{\mathrm{rev}}).
\]
Thus greedy records and block roots are two descriptions of the same
statistic under reflection.

Beyond the path generating identity, the results developed here are as
follows.
\begin{enumerate}[label=(\roman*)]
\item For every nonconstant orientation \(\eta\) of a cycle, a canonical root
chosen from the labeled cycle defines a cyclic record statistic for which
\[
\sum_{\pi\in S_n}t^{\operatorname{crec}_\eta(\pi)}
=n!\,\Omega(C_\eta;t).
\]
When the undirected Hasse diagram is a cycle, reflection identifies the cyclic
records with Kahane's circular block roots and proves his circular-fence
conjecture \cite[Conjecture~4.8]{Kahane}.
\item The Bernstein recurrence lifts the path identity to endpoint-refined
finite transfers and a direct recursive bijection for \(m\le n\), with an
algorithmic extension to arbitrary alphabets:
\begin{multline*}
S_n\times\{f:P_\eps\to[m]\text{ order-preserving}\}
\\[-2pt]
\longleftrightarrow\ 
\{(\pi,\lambda):\pi\in S_n,\ 
  \lambda:\operatorname{Rec}_\eps(\pi)\to[m]\}.
\end{multline*}
\item Refining by the complete record set, record directions, and terminal
value identifies every fixed fiber with decorated endpoint paths and with
pointed linear extensions of a record poset \(Q_{\eps,R}\), whose cover graph
is a caterpillar.  Atkinson's recursion for posets whose cover graph is a tree
\cite{AtkinsonTree}, specialized to record-gap coordinates, then gives the
terminal-value spectrum.  Standard \(P\)-partition theory supplies a further
quasisymmetric consequence.
\end{enumerate}

Related rank-polynomial work concerns fences \cite{MGO,KOR} and their circular
or loop analogues \cite{Chainlink,KLL}; rank-matrix models appear in
\cite{KantarciOguz}.  These polynomials enumerate lower ideals by cardinality.
By contrast, work on \(h^\ast\)-polynomials of zig-zag order and chain
polytopes \cite{CZ,CS} and on zig-zag Eulerian polynomials \cite{PZ} encodes
descent data.  Here the transfer tracks maps to arbitrary chains together with
their endpoint and complete record-set data.

After defining the statistic and comparing it pointwise with Kahane's blocks,
we prove the Bernstein transfer and the cyclic identity.  We then develop the
finite bijection, the record-set refinements, and their record-poset
consequences.

\section{Greedy records for oriented paths}\label{sec:greedy-paths}

\begin{definition}\label[definition]{def:general-fence}
Let \(n\ge1\) and \(\eps\in\{+,-\}^{n-1}\).  The \(\eps\)-fence \(P_\eps\) is
the poset on \(\{x_1,\ldots,x_n\}\) obtained as the transitive closure of the
relations
\[
\eps_i=+ \Longrightarrow x_i\succ x_{i+1},
\qquad
\eps_i=- \Longrightarrow x_i\prec x_{i+1}
\]
for \(1\le i\le n-1\).  Equivalently, its order-preserving maps to a chain are
exactly the maps satisfying
\[
\eps_i=+ \Longrightarrow f(x_i)\ge f(x_{i+1}),
\qquad
\eps_i=- \Longrightarrow f(x_i)\le f(x_{i+1})
\]
for \(1\le i\le n-1\).
\end{definition}

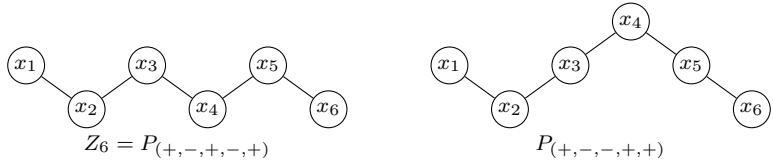
\begin{figure}[t]
\centering
\begin{tikzpicture}[x=.8cm,y=.58cm,baseline=(current bounding box.center),
  every node/.style={circle,draw,inner sep=1.4pt,font=\scriptsize}]
  \node (z1) at (0,1) {$x_1$};
  \node (z2) at (1,0) {$x_2$};
  \node (z3) at (2,1) {$x_3$};
  \node (z4) at (3,0) {$x_4$};
  \node (z5) at (4,1) {$x_5$};
  \node (z6) at (5,0) {$x_6$};
  \draw (z1)--(z2)--(z3)--(z4)--(z5)--(z6);
  \node[draw=none,rectangle] at (2.5,-.85)
    {$Z_6=P_{(+,-,+,-,+)}$};

  \node (p1) at (7,1) {$x_1$};
  \node (p2) at (8,0) {$x_2$};
  \node (p3) at (9,1) {$x_3$};
  \node (p4) at (10,2) {$x_4$};
  \node (p5) at (11,1) {$x_5$};
  \node (p6) at (12,0) {$x_6$};
  \draw (p1)--(p2)--(p3)--(p4)--(p5)--(p6);
  \node[draw=none,rectangle] at (9.5,-.85)
    {$P_{(+,-,-,+,+)}$};
\end{tikzpicture}
\caption{Two fence posets.  A \(+\)-edge is directed downward from left to
right, and a \(-\)-edge upward.}
\label{fig:fence-examples}
\end{figure}

\begin{definition}\label[definition]{def:records-chain}
Let \(\pi=\pi_1\cdots\pi_n\in S_n\).  The greedy \(\eps\)-record sequence of
\(\pi\) is the decreasing sequence of positions $i_0>i_1>\cdots>i_s$ constructed as follows.  Start with \(i_0=n\).  Once \(i_j\) has been chosen,
form
\[
A_j(\pi)=
\left\{i<i_j:
\begin{array}{ll}
\pi_i>\pi_{i_j}, & \text{if }\eps_i=+,\\
\pi_i<\pi_{i_j}, & \text{if }\eps_i=-
\end{array}
\right\}.
\]
If \(A_j(\pi)=\varnothing\), the construction stops. Otherwise, set $i_{j+1}=\max A_j(\pi)$. The positions \(i_0,i_1,\ldots,i_s\) are the \(\eps\)-records of \(\pi\).  Set
\[
\operatorname{Rec}_\eps(\pi)=\{i_0,i_1,\ldots,i_s\}.
\]
The record number is
\[
\rec_\eps(\pi)=|\operatorname{Rec}_\eps(\pi)|=s+1.
\]
We also set
\[
\rec_\eps^+(\pi)=\#\{1\le j\le s:\eps_{i_j}=+\},
\qquad
\rec_\eps^-(\pi)=\#\{1\le j\le s:\eps_{i_j}=-\}.
\]
Thus
\[
\rec_\eps(\pi)=1+\rec_\eps^+(\pi)+\rec_\eps^-(\pi),
\]
where the additional \(1\) comes from the terminal record at position \(n\),
which has no incoming sign.
\end{definition}

The greedy-chain definition is equivalent to the following threshold scan.

\begin{proposition}\label[proposition]{prop:chain-threshold-equivalence}
The sequence in \Cref{def:records-chain} is obtained by the following
right-to-left threshold scan.  Declare \(n\) to be a record and set
\(h=\pi_n\).  For \(i=n-1,n-2,\ldots,1\), declare \(i\) to be a record exactly
when \(\pi_i>h\) if \(\eps_i=+\), and exactly when \(\pi_i<h\) if
\(\eps_i=-\).  Whenever a record is declared, replace \(h\) by \(\pi_i\).
\end{proposition}

\begin{proof}
After \(i_j\) is chosen, the scan threshold is \(\pi_{i_j}\).  If
\(A_j(\pi)\ne\varnothing\), its largest element is the first position met from
the right that satisfies the required comparison, and both constructions
reset to its value.  If \(A_j(\pi)=\varnothing\), both stop.  Induction on
\(j\) proves the claim.
\end{proof}

\begin{proposition}[Reflection to the block statistic]
\label[proposition]{prop:kahane-path-comparison}
For \(\eps\in\{+,-\}^{n-1}\), define
\[
\delta_j=-\eps_{n-j}\qquad(1\le j\le n-1),
\qquad
\pi^{\mathrm{rev}}_j=\pi_{n+1-j}\qquad(1\le j\le n).
\]
Then \(P_\delta\) is the reflection of \(P_\eps\), and
\[
\rec_\eps(\pi)
=
\operatorname{bl}_{P_\delta}(\pi^{\mathrm{rev}}),
\]
where \(\operatorname{bl}\) is Kahane's greedy block statistic
\cite[Definition~3.5 and Proposition~3.4]{Kahane}.  More precisely, position
\(i\) is an \(\eps\)-record of \(\pi\) if and only if reflected position
\(n+1-i\) is the root of a block of \((P_\delta,\pi^{\mathrm{rev}})\).
\end{proposition}

\begin{proof}
Write \(y_j=x_{n+1-j}\).  The edge between \(y_j\) and \(y_{j+1}\) has sign
\(\delta_j=-\eps_{n-j}\), so the relabeling \(y_j\leftrightarrow x_{n+1-j}\)
identifies \(P_\delta\) with \(P_\eps\).

Kahane's greedy construction reads the reflected fence from left to right and
keeps the label of the leftmost vertex of the current block as its root
threshold.  At a reflected ascent, a new block begins exactly when the new
label is larger than that threshold; at a reflected descent, it begins exactly
when the new label is smaller.  For \(j\ge2\), put \(i=n+1-j\).  The vertex
\(y_j\) is an ascent when \(\delta_{j-1}=-\), equivalently when
\(\eps_i=+\), and it is a descent when \(\delta_{j-1}=+\), equivalently when
\(\eps_i=-\).  Hence a new block begins at \(j\) exactly when
\[
\eps_i=+\text{ and }\pi_i>h,
\qquad\text{or}\qquad
\eps_i=-\text{ and }\pi_i<h,
\]
where \(h\) is the current root label.  These are precisely the two record
conditions in \Cref{prop:chain-threshold-equivalence}.  Both scans start at the
same reflected endpoint and reset the threshold at the same positions, which
proves the pointwise correspondence and the equality of the two statistics.
\end{proof}

\begin{example}\label[example]{ex:41532}
For the zig-zag sign sequence
\(\eps^{\mathrm{zz}}\) of length \(n-1\), given by
\(\eps_i^{\mathrm{zz}}=+\) for odd \(i\) and \(\eps_i^{\mathrm{zz}}=-\) for
even \(i\), and for \(\pi=41532\), the greedy record sequence is
\[
5\longrightarrow 3\longrightarrow 2\longrightarrow 1.
\]
Starting from position \(5\), whose value is \(2\), position \(4\) is even and
therefore has sign \(-\), so it would have to satisfy \(\pi_4<2\); but
\(\pi_4=3\).  The nearest admissible position is therefore \(3\), since
\(\pi_3=5>2\).  From position \(3\), the next admissible position is \(2\), and
from position \(2\), the next admissible position is \(1\).  Hence
\(\rec_{\eps^{\mathrm{zz}}}(41532)=4\), which is \(\zzrec(41532)\) in the
notation used below.
\end{example}

\section{The Bernstein transfer proof}\label{sec:bernstein-proof}

The proof of \Cref{thm:main-general} compares two explicit models for the same
right-to-left process.  The first model uses polynomial recurrences
\(H_w^{(m)}(x)\in\mathbb Q[x]\) derived from a continuous record process.  The
second model uses endpoint-refined order-preserving maps.  The bridge between
the two is the Bernstein basis.

Let \(w=w_1\cdots w_r\in\{+,-\}^r\) be a word.  The letter \(w_1\) is the first
comparison made with the current right endpoint, \(w_2\) is the next comparison
after moving one step to the left, and so on.

\begin{definition}\label[definition]{def:endpoint-counts}
Fix \(m\ge1\).  For \(1\le k\le m\), define endpoint-refined counts
\(C_w^{(m)}(k)\) recursively by $C_\varnothing^{(m)}(k)=1$, and for a word \(v\),
\[
C_{+v}^{(m)}(k)=\sum_{\ell=k}^{m} C_v^{(m)}(\ell),
\qquad
C_{-v}^{(m)}(k)=\sum_{\ell=1}^{k} C_v^{(m)}(\ell).
\]
Equivalently, \(C_w^{(m)}(k)\) counts sequences
\((g_0,g_1,\ldots,g_r)\in[m]^{r+1}\) with \(g_0=k\) such that
\[
w_j=+ \Longrightarrow g_j\ge g_{j-1},
\qquad
w_j=- \Longrightarrow g_j\le g_{j-1}
\]
for \(1\le j\le r\).
\end{definition}

\begin{lemma}[Endpoint-refined order polynomials]\label[lemma]{lem:endpoint-order}
Let \(\eps\in\{+,-\}^{n-1}\), and put $w=\eps_{n-1}\eps_{n-2}\cdots\eps_1$. For \(1\le k\le m\), the number \(C_w^{(m)}(k)\) counts order-preserving maps
\(f:P_\eps\to[m]\) with \(f(x_n)=k\).  Consequently,
\[
\Omega(P_\eps;m)=\sum_{k=1}^m C_w^{(m)}(k).
\]
\end{lemma}

\begin{proof}
Given \(f:P_\eps\to[m]\), set \(g_j=f(x_{n-j})\) for \(0\le j\le n-1\).  Then
\(g_0=f(x_n)\).  For \(1\le j\le n-1\), the sign \(w_j\) is
\(\eps_{n-j}\).  If \(w_j=+\), then the order-preserving condition is $f(x_{n-j})\ge f(x_{n-j+1})$, or \(g_j\ge g_{j-1}\).  If \(w_j=-\), it is \(g_j\le g_{j-1}\).  These are exactly the inequalities in \Cref{def:endpoint-counts}, and the correspondence is reversible.
\end{proof}

\begin{definition}\label[definition]{def:continuous-H}
For \(m\ge1\), \(x\in[0,1]\), and a word \(w\in\{+,-\}^r\), let
\(H_w^{(m)}(x)\) be the expected value of \(m^N\) in the following process.  We
start with threshold \(x\), inspect \(r\) independent uniform random variables
in order, and use the comparison directions \(w_1,\ldots,w_r\).  Whenever the
current direction is \(+\), a new record occurs if the inspected value is larger
than the current threshold; whenever the current direction is \(-\), a new
record occurs if the inspected value is smaller than the current threshold.  On
a record, the threshold is replaced by the inspected value; otherwise it is
unchanged.  The random variable \(N\) is the number of new records among these
\(r\) inspected variables.
\end{definition}

Conditioning on the first inspected value gives the recurrences
\begin{align}
H_\varnothing^{(m)}(x)&=1, \label{eq:H-empty}\\
H_{+v}^{(m)}(x)
&=xH_v^{(m)}(x)+m\int_x^1H_v^{(m)}(u)\dd u, \label{eq:H-plus}\\
H_{-v}^{(m)}(x)
&=(1-x)H_v^{(m)}(x)+m\int_0^xH_v^{(m)}(u)\dd u. \label{eq:H-minus}
\end{align}
For example, in a \(+\)-step, a first inspected value \(u\le x\) is not a
record, keeps the threshold \(x\), and contributes \(xH_v^{(m)}(x)\).  A value
\(u>x\) is a record, contributes the factor \(m\), changes the threshold to
\(u\), and contributes \(m\int_x^1H_v^{(m)}(u)\dd u\).  The \(-\)-step is the
same argument with the intervals \([0,x]\) and \([x,1]\) interchanged.
These recurrences recursively determine \(H_w^{(m)}\) from the constant
polynomial \(1\).  Hence each \(H_w^{(m)}\) is a polynomial in \(x\), and all
integrations are ordinary integrations of polynomials.

For \(1\le k\le m\), set
\[
b_{m,k}(x)=\binom{m-1}{k-1}x^{k-1}(1-x)^{m-k}.
\]
These are the Bernstein basis polynomials of degree \(m-1\)
\cite{Farouki}, indexed so that \(k\in[m]\) matches the endpoint value in
\Cref{lem:endpoint-order}.  Probabilistically, if
\(U_1,\ldots,U_{m-1}\) are independent uniform random variables on \([0,1]\),
then
\[
b_{m,k}(x)=
\Pr\bigl(\#\{j:U_j<x\}=k-1\bigr).
\]
Thus \(b_{m,k}(x)\) is the probability that \(x\) has rank \(k\) after being
inserted among \(m-1\) auxiliary uniform points.

\begin{lemma}[Bernstein summation identities]\label[lemma]{lem:bernstein-identities}
For \(1\le\ell\le m\),
\begin{align}
x b_{m,\ell}(x)+m\int_x^1 b_{m,\ell}(u)\dd u
&=\sum_{k=1}^{\ell} b_{m,k}(x), \label{eq:bernstein-plus}\\
(1-x)b_{m,\ell}(x)+m\int_0^x b_{m,\ell}(u)\dd u
&=\sum_{k=\ell}^{m} b_{m,k}(x). \label{eq:bernstein-minus}
\end{align}
Moreover,
\[
\int_0^1 b_{m,k}(x)\dd x=\frac1m
\]
for every \(k\).
\end{lemma}

\begin{proof}
For the first identity, differentiate both sides: each derivative equals
\[
-(m-\ell)\binom{m-1}{\ell-1}
x^{\ell-1}(1-x)^{m-\ell-1}
\]
when \(\ell<m\), and both sides vanish at \(x=1\); the case \(\ell=m\) is
the binomial theorem.  The second identity follows by replacing
\((x,\ell)\) with \((1-x,m+1-\ell)\).  Finally, the beta integral gives
\[
\int_0^1 b_{m,k}(x)\dd x
=
\binom{m-1}{k-1}\frac{(k-1)!(m-k)!}{m!}
=\frac1m,
\].
\end{proof}

\begin{lemma}[Transfer lemma]\label[lemma]{lem:transfer}
For every \(m\ge1\) and every word \(w\in\{+,-\}^r\),
\[
H_w^{(m)}(x)=\sum_{k=1}^{m} C_w^{(m)}(k)b_{m,k}(x).
\]
\end{lemma}

\begin{proof}
We induct on the length of \(w\).  For \(w=\varnothing\), the statement is $1=\sum_{k=1}^m b_{m,k}(x)$, the binomial theorem.

Assume the statement for \(v\).  Using \eqref{eq:H-plus}, the induction
hypothesis, and \eqref{eq:bernstein-plus}, we get
\[
\begin{aligned}
H_{+v}^{(m)}(x)&=x\sum_{\ell=1}^mC_v^{(m)}(\ell)b_{m,\ell}(x)+m\int_x^1\sum_{\ell=1}^mC_v^{(m)}(\ell)b_{m,\ell}(u)\dd u\\
&=\sum_{\ell=1}^m C_v^{(m)}(\ell)
\left(xb_{m,\ell}(x)+m\int_x^1b_{m,\ell}(u)\dd u\right)\\
&=\sum_{\ell=1}^m C_v^{(m)}(\ell)\sum_{k=1}^{\ell} b_{m,k}(x)\\
&=\sum_{k=1}^m\left(\sum_{\ell=k}^m C_v^{(m)}(\ell)\right)b_{m,k}(x)\\
&=\sum_{k=1}^m C_{+v}^{(m)}(k)b_{m,k}(x).
\end{aligned}
\]
The proof for \(-v\) is identical, using \eqref{eq:H-minus} and
\eqref{eq:bernstein-minus}:
\[
\begin{aligned}
H_{-v}^{(m)}(x)&=(1-x)\sum_{\ell=1}^mC_v^{(m)}(\ell)b_{m,\ell}(x)+m\int_0^x\sum_{\ell=1}^mC_v^{(m)}(\ell)b_{m,\ell}(u)\dd u\\
&=\sum_{\ell=1}^m C_v^{(m)}(\ell)
\left((1-x)b_{m,\ell}(x)+m\int_0^xb_{m,\ell}(u)\dd u\right)\\
&=\sum_{\ell=1}^m C_v^{(m)}(\ell)\sum_{k=\ell}^{m} b_{m,k}(x)\\
&=\sum_{k=1}^m\left(\sum_{\ell=1}^k C_v^{(m)}(\ell)\right)b_{m,k}(x)\\
&=\sum_{k=1}^m C_{-v}^{(m)}(k)b_{m,k}(x).
\end{aligned}
\]
This completes the induction.
\end{proof}

\begin{proof}[Proof of \Cref{thm:main-general}]
It suffices to prove the identity after the specialization \(t=m\) for every
positive integer \(m\), since both sides are polynomials in \(t\).

Let \(X_1,\ldots,X_n\) be independent continuous uniform random variables.
Their relative order is almost surely uniformly distributed over \(S_n\), and
\(\rec_\eps\) depends only on comparisons among entries.  Therefore
\[
\frac1{n!}\sum_{\pi\in S_n}m^{\rec_\eps(\pi)}
\]
is the expected value of \(m^{\rec_\eps}\) in this continuous model.

Set $w=\eps_{n-1}\eps_{n-2}\cdots\eps_1$. The rightmost position is always a record, contributing one factor \(m\).  After
conditioning on \(X_n=x\), the remaining variables are inspected with comparison
word \(w\).  Hence
\begin{equation}\label{eq:random-side}
\frac1{n!}\sum_{\pi\in S_n}m^{\rec_\eps(\pi)}
=m\int_0^1 H_w^{(m)}(x)\dd x.
\end{equation}

By \Cref{lem:transfer} and \Cref{lem:bernstein-identities},
\[
m\int_0^1 H_w^{(m)}(x)\dd x
=
m\sum_{k=1}^m C_w^{(m)}(k)\int_0^1b_{m,k}(x)\dd x
=
\sum_{k=1}^m C_w^{(m)}(k).
\]
By \Cref{lem:endpoint-order}, the last sum is \(\Omega(P_\eps;m)\).  Combining
this with \eqref{eq:random-side} proves
\[
\frac1{n!}\sum_{\pi\in S_n}m^{\rec_\eps(\pi)}=\Omega(P_\eps;m)
\]
for all positive integers \(m\).  The polynomial identity in \(t\) follows.
\end{proof}

For the alternating sign word
\(\eps_i^{\mathrm{zz}}=+\) for odd \(i\) and
\(\eps_i^{\mathrm{zz}}=-\) for even \(i\), write
\[
\zzrec(\pi)=\rec_{\eps^{\mathrm{zz}}}(\pi).
\]
Then \(P_{\eps^{\mathrm{zz}}}=Z_n\), and \Cref{thm:main-general} specializes
to
\[
\sum_{\pi\in S_n}t^{\zzrec(\pi)}=n!\,\Omega(Z_n;t).
\]
This is the zig-zag specialization posed in \cite[Problem~5.3]{FMP}.

\begin{remark}[Symmetries and checks]\label[remark]{cor:sign-reversal}
\label{rem:path-reversal}
Complementing values, \(c(\pi)_i=n+1-\pi_i\), interchanges the two record
directions:
\begin{align*}
\rec_{-\eps}(c(\pi))&=\rec_\eps(\pi),\\
\rec_{-\eps}^+(c(\pi))&=\rec_\eps^-(\pi),\\
\rec_{-\eps}^-(c(\pi))&=\rec_\eps^+(\pi).
\end{align*}
This follows directly by complementing the threshold at every step.  If
\(\eps^\leftarrow_i=\eps_{n-i}\), reversing the vertices identifies
\(P_\eps\) with \(P_\eta\), where
\(\eta_i=-\eps_{n-i}\).  Complementing the values of an order-preserving map
then identifies \(P_\eta\) with \(P_{-\eta}\), so
\[
\Omega(P_\eps;t)=\Omega(P_{\eps^\leftarrow};t).
\]
Thus sign reversal and word reversal do not produce new order polynomials.

For the monotone word \((+,\ldots,+)\), the statistic is the classical number
of right-to-left maxima and
\[
\sum_{\pi\in S_n}t^{\rec_\eps(\pi)}
=n!\binom{t+n-1}{n}=t(t+1)\cdots(t+n-1).
\]
The all-\(-\) case follows by complementation.

For the alternating sign word, the covering relations are
\[
x_1\succ x_2\prec x_3\succ x_4\prec\cdots .
\]
Under the opposite convention
\(x_1\prec x_2\succ x_3\prec\cdots\), the signs in
\(\eps^{\mathrm{zz}}\) are reversed, so the two statistics are
equidistributed by value complementation.  Every position is a record
exactly for a down-up alternating permutation; hence, as a consistency check,
\[
[t^n]\sum_{\pi\in S_n}t^{\zzrec(\pi)}=E_n,
\]
the classical Euler zig-zag number \cite{StanleySurvey}.
\end{remark}

\section{Circular fences and cyclic records}\label{sec:circular-records}

Let \(\eta=(\eta_1,\ldots,\eta_n)\in\{+,-\}^n\) be nonconstant, and read
subscripts modulo \(n\).  The cyclic orientation poset \(C_\eta\) is the
transitive closure of
\[
\eta_i=+ \Longrightarrow x_i\succ x_{i+1},
\qquad
\eta_i=- \Longrightarrow x_i\prec x_{i+1}.
\]
Nonconstancy makes this orientation acyclic and ensures that both signs occur.
Thus \(C_\eta\) is a poset, and its order-preserving maps to \([m]\) are the
cyclic sequences satisfying the corresponding weak inequalities.  Rank-matrix
and trace models for related oriented and circular fences were developed by
Kantarc{\i} O\u{g}uz \cite{KantarciOguz}.
If one sign occurs only once, its edge is redundant in the transitive closure
and the Hasse diagram is a path.  If both signs occur at least twice, the
undirected Hasse diagram is a cycle and \(C_\eta\) is a circular fence in the
usual sense.  The record identity below includes both cases; the comparison
with Kahane's circular blocks applies only in the latter.

\begin{proposition}[Trace formula]\label[proposition]{prop:circular-trace}
For a positive integer \(m\), let
\[
M_+[a,b]=\mathbf 1_{a\ge b},
\qquad
M_-[a,b]=\mathbf 1_{a\le b}
\qquad(a,b\in[m]).
\]
Then
\[
\Omega(C_\eta;m)=\operatorname{tr}(M_{\eta_1}\cdots M_{\eta_n}).
\]
\end{proposition}

\begin{proof}
The \((a_1,a_{n+1})\)-entry of the product sums the indicators of the edge
inequalities over \(a_2,\ldots,a_n\in[m]\).  Taking the trace imposes
\(a_{n+1}=a_1\), and hence counts the order-preserving maps.  This is the
standard transfer-matrix argument; compare \cite[\S4.7]{StanleyEC1}.  For an
analogous trace construction for circular-fence rank polynomials, see
\cite[Proposition~5.3]{KantarciOguz}.
\end{proof}

Ferroni, Morales, and Panova proved coefficient nonnegativity for every circular
fence using the Gessel--Krattenthaler determinant for cylindric partitions
\cite[Theorem~7.7]{FMP}, based on the determinant of Gessel and Krattenthaler
\cite{GesselKrattenthaler}.  For closed alternating zig-zags,
Lundstr\"om and Saud Maia Leite give a cyclic-swap interpretation of the
\(h^\ast\)-polynomial \cite{LSL}.  Here we seek instead a statistic for the
coefficients of \(n!\Omega(C_\eta;t)\).  The trace formula admits such a record
interpretation once the missing endpoint is replaced by a canonical root.

\begin{definition}[Cyclic right-to-left records]\label[definition]{def:circular-records}
Let $I_+(\eta)=\{i\in[n]:\eta_i=+\}$. For \(\pi=\pi_1\cdots\pi_n\in S_n\), let \(r=r_\eta(\pi)\) be the unique
index in \(I_+(\eta)\) such that $\pi_r=\max\{\pi_i:i\in I_+(\eta)\}$. The position \(r\) is declared to be a record, and the initial threshold is
\(h=\pi_r\).  We then inspect the remaining positions in cyclic right-to-left
order $r-1,\ r-2,\ \ldots,\ r+1$. When position \(i\) is inspected, it is a record if
\[
\eta_i=+\text{ and }\pi_i>h,
\qquad\text{or}\qquad
\eta_i=-\text{ and }\pi_i<h.
\]
Whenever a record occurs, the threshold is reset to \(h=\pi_i\).  The total
number of records is denoted by \(\operatorname{crec}_\eta(\pi)\).
\end{definition}

Thus the terminal position used for paths is replaced by the largest value
sitting at the tail of a \(+\)-edge.  This choice is intrinsic to the labeled
cycle and is almost surely unique in the continuous model used below.  Rooting
at a \(+\)-edge maximum is a convention: value complementation together with
sign reversal exchanges it with rooting at the smallest value on a
\(-\)-edge.

\begin{theorem}[Circular records]\label{thm:circular-records}
For every nonconstant \(\eta\in\{+,-\}^n\), we have
\[
\sum_{\pi\in S_n}t^{\operatorname{crec}_\eta(\pi)}
=
n!\,\Omega(C_\eta;t).
\]
\end{theorem}

\begin{proof}
It suffices to prove the identity after substituting \(t=m\) for every positive
integer \(m\), since both sides are polynomials in \(t\).

Let \(X_1,\ldots,X_n\) be independent uniform random variables on \([0,1]\).
Their relative order is almost surely uniform on \(S_n\), and
\(\operatorname{crec}_\eta\) depends only on this relative order.  Hence
\[
\frac1{n!}\sum_{\pi\in S_n}m^{\operatorname{crec}_\eta(\pi)}
=
\mathbb E\bigl[m^{\operatorname{crec}_\eta(X)}\bigr].
\]

Let \(A_+\) and \(A_-\) be the \(m\times m\) matrices
\[
A_+[a,b]=
\begin{cases}
1,&a\le b,\\
0,&a>b,
\end{cases}
\qquad
A_-[a,b]=
\begin{cases}
1,&a\ge b,\\
0,&a<b,
\end{cases}
\]
with rows and columns indexed by \([m]\).  These are the one-step transfer
matrices for the endpoint recurrences in \Cref{def:endpoint-counts}: a
\(+\)-step sends a continuation vector \(v\) to \(A_+v\), and a \(-\)-step
sends it to \(A_-v\).  Put
\[
N_+(x)=xA_+,
\qquad
N_-(x)=(1-x)I+xA_-.
\]
Let \(e_m\) be the last standard basis vector and let \(\mathbf 1\) be the
all-one column vector.  These are the boundary states of the transfer.  After
conditioning on a root value \(x\) and rescaling \([0,x]\) to \([0,1]\), the
initial threshold is \(1\).  Since \(b_{m,k}(1)=\delta_{k,m}\), the Bernstein
expansion in \Cref{lem:transfer} evaluates this state by \(e_m^T\).  At the
other end the empty continuation has coefficient vector \(\mathbf 1\), because
\(C_\varnothing^{(m)}(k)=1\) for all \(k\).  For \(r\in I_+(\eta)\), define
\[
T_r(x)=N_{\eta_{r-1}}(x)N_{\eta_{r-2}}(x)\cdots
N_{\eta_{r+1}}(x),
\]
where the product contains the \(n-1\) indices met in cyclic decreasing
order.
We claim that
\begin{equation}\label{eq:circular-root-decomposition}
\frac1{n!}\sum_{\pi\in S_n}m^{\operatorname{crec}_\eta(\pi)}
=
m\int_0^1\sum_{r\in I_+(\eta)} e_m^T T_r(x)\mathbf 1\,\dd x .
\end{equation}
Indeed, fix \(r\in I_+(\eta)\) and condition on \(X_r=x\).  The contribution in
which \(r\) is the chosen root requires every other \(+\)-position to have value
in \([0,x]\).  During the cyclic scan the threshold never exceeds \(x\).  Thus a
non-root \(+\)-position contributes the probability factor \(x\) and, after
rescaling \([0,x]\) to \([0,1]\), the ordinary \(+\)-transfer \(A_+\).  A
\(-\)-position has two possibilities: values in \((x,1]\) cannot be records and
contribute \((1-x)I\), while values in \([0,x]\) contribute, after the same
rescaling, the ordinary \(-\)-transfer \(xA_-\).  The root itself contributes
one factor \(m\).  Multiplying the transfers in the scan order and integrating
over \(x\) gives \eqref{eq:circular-root-decomposition}.

To evaluate the integral, set
\[
\mathcal N_\eta(x)
=N_{\eta_n}(x)N_{\eta_{n-1}}(x)\cdots N_{\eta_1}(x).
\]
We shall show that
\begin{equation}\label{eq:circular-trace-derivative}
\frac{\dd}{\dd x}\operatorname{tr}\mathcal N_\eta(x)
=
m\sum_{r\in I_+(\eta)}e_m^T T_r(x)\mathbf 1.
\end{equation}
Let $R=m\mathbf 1 e_m^T$ and let \(L\) be the lower bidiagonal matrix
\[
L_{a,a}=a-m\quad(1\le a\le m),
\qquad
L_{a,a-1}=-(a-1)\quad(2\le a\le m),
\]
with all other entries equal to \(0\).  A direct entrywise check gives
\[
LA_- - A_-L=A_- - I,
\qquad
LA_+ - A_+L=A_+ - R.
\]
Indeed, multiplication by the two nonzero diagonals of \(L\) gives
\[
\begin{aligned}
(LA_--A_-L)[a,b]
 &=\mathbf 1_{a\ge b}-\delta_{a,b},\\
(LA_+-A_+L)[a,b]
 &=\mathbf 1_{a\le b}-m\mathbf 1_{b=m},
\end{aligned}
\]
which are the entries of \(A_--I\) and
\(A_+-m\mathbf 1e_m^T\), respectively.
Equivalently, for \(x>0\), with \(D(x)=L/x\),
\[
N_-'(x)=[D(x),N_-(x)],
\qquad
N_+'(x)=R+[D(x),N_+(x)].
\]
Differentiating \(\operatorname{tr}\mathcal N_\eta(x)\), the Leibniz sum of
the commutator terms telescopes by the derivation rule:
\[
\sum_{i=1}^{n}
 N_{\eta_n}\cdots N_{\eta_{i+1}}[D,N_{\eta_i}]
 N_{\eta_{i-1}}\cdots N_{\eta_1}
=[D,\mathcal N_\eta].
\]
(Here every matrix is evaluated at \(x\).)  Its trace is zero.  The remaining
terms insert \(R\) at the \(+\)-positions.  By cyclicity of trace, these
insertions give
\[
\sum_{r\in I_+(\eta)}
\operatorname{tr}\left(RT_r(x)\right)
=
m\sum_{r\in I_+(\eta)}e_m^T T_r(x)\mathbf 1,
\]
which proves \eqref{eq:circular-trace-derivative}.  Although \(D(x)\) has a
pole at \(0\), the displayed derivative identity is an identity of polynomials
on \((0,1]\), hence extends to the interval after integration.

Since \(I_+(\eta)\ne\varnothing\), we have \(N_+(0)=0\) in at least one factor,
so \(\operatorname{tr}\mathcal N_\eta(0)=0\).  Also
\(N_+(1)=A_+\) and \(N_-(1)=A_-\).  Integrating
\eqref{eq:circular-trace-derivative} over \([0,1]\) gives
\[
m\int_0^1\sum_{r\in I_+(\eta)} e_m^T T_r(x)\mathbf 1\,\dd x
=
\operatorname{tr}\left(A_{\eta_n}A_{\eta_{n-1}}\cdots A_{\eta_1}\right).
\]
The matrices \(A_+\) and \(A_-\) are the transposes of the matrices \(M_+\) and
\(M_-\) in \Cref{prop:circular-trace}.  Therefore
\[
\operatorname{tr}\left(A_{\eta_n}A_{\eta_{n-1}}\cdots A_{\eta_1}\right)
=
\operatorname{tr}\left(M_{\eta_1}M_{\eta_2}\cdots M_{\eta_n}\right)
=
\Omega(C_\eta;m).
\]
Combining this with \eqref{eq:circular-root-decomposition} proves the desired
identity at \(t=m\), and polynomial interpolation completes the proof.
\end{proof}

We next compare the cyclic records with Kahane's circular blocks.  This
comparison requires the undirected Hasse diagram to be a cycle.

\begin{lemma}[Closing a greedy block path]
\label{lem:closing-greedy-blocks}
Let \(C\) be a circular fence whose Hasse diagram is a cycle, and let \(y_s\)
be the ascent vertex with largest label among all ascent vertices.  Delete
the incoming cover \(y_{s-1}<y_s\), and read the resulting path from \(y_s\)
to \(y_{s-1}\).  If \(B_1,\ldots,B_k\) is Kahane's greedy block partition of
this path, then the same consecutive blocks form the unique valid circular
block partition of \(C\).
\end{lemma}

\begin{proof}
The first block has root \(y_s\).  Every later ascent root has smaller label
by the choice of \(y_s\), while a descent root is created by a downward
threshold reset and therefore also has smaller label.  If \(k=1\), all other
ascent labels lie below the label of \(y_s\), and every descent label lies
above it.  Thus the whole cycle satisfies Kahane's block condition.

Suppose \(k\ge2\).  The ascent case in the proof of
\cite[Lemma~4.7]{Kahane} deletes the incoming cover at an ascent vertex and
extends the resulting path partition back to the cycle.  Inspection of that
argument shows that it uses maximality of the chosen label only to ensure that
the chosen root has larger label than the root of the block containing its
predecessor.  Here that predecessor lies in \(B_k\), and the required
inequality follows from the first paragraph.  Thus the path blocks are
admissible circular blocks after the cover is restored.

For completeness, validity between blocks can also be checked directly.  All
relations not using the restored cover were already present in the path block
poset and are valid.  The restored cover gives the additional block relation
\(B_k<B_1\), and the root of \(B_k\) has smaller label than the root \(y_s\) of
\(B_1\).  Any new transitive relation between blocks is represented by a
directed path in the block quotient that uses this additional edge.  Root
labels increase along every edge of that path, so every new relation satisfies
Kahane's validity condition.  The restored partition is therefore valid, and
its uniqueness follows again from \cite[Lemma~4.7]{Kahane}.
\end{proof}

\begin{proposition}[Reflection and circular blocks]
\label[proposition]{prop:kahane-circular-comparison}
Assume that both signs occur at least twice in \(\eta\), equivalently that
the Hasse diagram of \(C_\eta\) is a cycle.
Read subscripts modulo \(n\), with \(\eta_0=\eta_n\), and define
\[
\delta_j=-\eta_{n-j},
\qquad
\pi^{\mathrm{rev}}_j=\pi_{n+1-j}
\qquad(1\le j\le n).
\]
Let \(\operatorname{bl}^{\circ}_{C_\delta}\) denote Kahane's
\(\widehat{\operatorname{bl}}\), the number of blocks in the unique valid
circular block partition \cite[\S4.2 and Lemma~4.7]{Kahane}.  Then
\[
\operatorname{crec}_\eta(\pi)
=
\operatorname{bl}^{\circ}_{C_\delta}(\pi^{\mathrm{rev}}).
\]
Under this correspondence, the canonical cyclic record becomes the root with
largest label among the ascent vertices of the reflected cycle, and every
subsequent cyclic record becomes the root of the next block.
\end{proposition}

\begin{proof}
Put \(y_j=x_{n+1-j}\).  As in
\Cref{prop:kahane-path-comparison}, the signs of the reflected cycle are
\(\delta_j=-\eta_{n-j}\).  Moreover, \(y_{n+1-i}=x_i\) is an ascent vertex of
the reflected cycle exactly when \(i\in I_+(\eta)\).  The chosen cyclic record
therefore becomes the maximum-labeled ascent vertex, say \(y_s\).

Cut the edge entering \(y_s\) and read the resulting path from \(y_s\) in the
forward cyclic direction.  This is the reflection of the cyclic right-to-left
scan.  The local comparison in Kahane's greedy block construction is the same
as in the proof of \Cref{prop:kahane-path-comparison}: at an ascent a new block
starts when the new label is larger than the current root label, and at a
descent it starts when the new label is smaller.  Hence the block roots of the
cut path are exactly the reflected cyclic records.

By \Cref{lem:closing-greedy-blocks}, restoring the cut edge produces Kahane's
unique valid circular partition without changing the block roots.  The
pointwise equality follows.
\end{proof}

\begin{corollary}[Kahane's circular-fence conjecture]
\label[corollary]{cor:kahane-circular-conjecture}
Let \(C\) be a circular fence of size \(n\) whose Hasse diagram is a cycle,
in the sense of \cite[Definition~4.6]{Kahane}.  Then
\[
n!\,\Omega(C;t)
=
\sum_{\sigma\in S_n}t^{\operatorname{bl}^{\circ}_{C}(\sigma)}.
\]
Consequently, \cite[Conjecture~4.8]{Kahane} holds.
\end{corollary}

\begin{proof}
Choose a cyclic presentation \(C=C_\delta\) and reflect it to
\(C_\eta\), where \(\eta_i=-\delta_{n-i}\).  Reflection is a bijection on
labelings and preserves the order polynomial.  The result follows by combining
\Cref{prop:kahane-circular-comparison} with
\Cref{thm:circular-records}.
\end{proof}

\begin{example}
For the four-vertex crown \(\eta=(+,-,+,-)\),
\[
\Omega(C_\eta;m)
=\sum_{a,c=1}^m\min(a,c)^2
=\frac{m(m+1)(m^2+m+1)}6.
\]
The cyclic record statistic therefore gives
\[
\sum_{\pi\in S_4}t^{\operatorname{crec}_\eta(\pi)}
=
4t+8t^2+8t^3+4t^4
=
4!\,\Omega(C_\eta;t).
\]
\end{example}

\section{Finite Bernstein transfers}\label{sec:finite-transfer}

We now discretize the Bernstein transfer.  Finite counting recurrences for a
rank-vector analogue of \(H_w^{(m)}\) give a transfer between rank-vector sums
and endpoint counts.  Its diagonal case is realized bijectively in
\Cref{sec:direct-bijection}.

\begin{definition}[Rank-vector record sums]\label[definition]{def:rank-vector-D}
Let \(w=w_1\cdots w_r\in\{+,-\}^r\), let \(m\ge1\), and let
\(1\le k\le r+1\).  For a permutation
\(\alpha=(a_1,\ldots,a_r)\) of the set \([r+1]\setminus\{k\}\), start with
threshold \(h_0=k\).  At step \(j\), declare a record if
\[
w_j=+\text{ and }a_j>h_{j-1},
\qquad\text{or}\qquad
w_j=-\text{ and }a_j<h_{j-1}.
\]
If a record occurs, set \(h_j=a_j\); otherwise set \(h_j=h_{j-1}\).  Let
\(N_w(\alpha;k)\) be the number of records declared in these \(r\) steps, not
including the initial threshold.  Define
\[
D_w^{(m)}(k)=\sum_{\alpha}m^{N_w(\alpha;k)},
\]
where the sum is over all permutations of \([r+1]\setminus\{k\}\).  For
indices outside \(1,\ldots,r+1\), we set \(D_w^{(m)}(k)=0\).
\end{definition}

\begin{lemma}[Entry from permutations]\label[lemma]{lem:finite-entry}
Let \(\eps\in\{+,-\}^{n-1}\), put $w=\eps_{n-1}\eps_{n-2}\cdots\eps_1$, and let \(1\le k\le n\).  Then, for every positive integer \(m\),
\[
\sum_{\substack{\pi\in S_n\\ \pi_n=k}}m^{\rec_\eps(\pi)}
=m\,D_w^{(m)}(k).
\]
Consequently,
\[
\sum_{\pi\in S_n}m^{\rec_\eps(\pi)}
=m\sum_{k=1}^nD_w^{(m)}(k).
\]
\end{lemma}

\begin{proof}
The position \(n\) is always a greedy \(\eps\)-record, so it contributes one
factor \(m\).  If \(\pi_n=k\), then the right-to-left scan of the remaining
positions reads a permutation of \([n]\setminus\{k\}\), and the comparison word
is \(w=\eps_{n-1}\cdots\eps_1\).  The rule in
\Cref{def:rank-vector-D} is exactly the threshold scan of
\Cref{prop:chain-threshold-equivalence}, with the terminal record removed.
Summing over all permutations with fixed terminal value \(k\) gives the first
identity, and summing over \(k\) gives the second.
\end{proof}

\begin{proposition}[Rank-vector recurrences]\label[proposition]{prop:rank-vector-recurrences}
Let \(v\in\{+,-\}^{r-1}\), so that \(+v\) and \(-v\) have length \(r\).  For
\(1\le k\le r+1\),
\begin{align}
D_{+v}^{(m)}(k)
&=(k-1)D_v^{(m)}(k-1)
  +m\sum_{j=k}^{r}D_v^{(m)}(j), \label{eq:D-plus}\\
D_{-v}^{(m)}(k)
&=(r+1-k)D_v^{(m)}(k)
  +m\sum_{j=1}^{k-1}D_v^{(m)}(j). \label{eq:D-minus}
\end{align}
\end{proposition}

\begin{proof}
For \(D_{+v}^{(m)}(k)\), separate the first scanned value \(a\).  If \(a<k\),
then no record occurs.  After deleting \(a\) and standardizing the remaining
ordered set, the threshold \(k\) has rank \(k-1\).  There are \(k-1\) choices
for \(a\), giving the first term in \eqref{eq:D-plus}.  If \(a>k\), then a
record occurs, contributing the factor \(m\).  After deleting the old threshold
\(k\) and standardizing the ordered set \([r+1]\setminus\{k\}\), the new
threshold \(a\) has rank \(j=a-1\), where \(j\) ranges from \(k\) to \(r\).
This gives the sum in \eqref{eq:D-plus}.

For \(D_{-v}^{(m)}(k)\), the same first-step decomposition is reversed.  If
\(a>k\), no record occurs; deleting \(a\) leaves the threshold with rank \(k\),
and there are \(r+1-k\) choices.  If \(a<k\), a record occurs, contributes
\(m\), and the new threshold has rank \(j=a\) after standardization, with
\(1\le j<k\).  This gives \eqref{eq:D-minus}.
\end{proof}

For \(r\ge0\), write
\[
B_{r,k}(x)=\binom r{k-1}x^{k-1}(1-x)^{r+1-k},
\qquad 1\le k\le r+1,
\]
for the degree \(r\) Bernstein basis.

\begin{lemma}[Degree-raising identities]\label[lemma]{lem:degree-raising}
Let \(r\ge1\) and \(1\le j\le r\).  Then
\begin{align}
xB_{r-1,j}(x)&=\frac{j}{r}B_{r,j+1}(x),\nonumber\\
\int_x^1B_{r-1,j}(u)\dd u&=\frac1r\sum_{h=1}^{j}B_{r,h}(x),
\label{eq:degree-raising-plus}\\
(1-x)B_{r-1,j}(x)&=\frac{r+1-j}{r}B_{r,j}(x),\nonumber\\
\int_0^xB_{r-1,j}(u)\dd u&=\frac1r\sum_{h=j+1}^{r+1}B_{r,h}(x).
\label{eq:degree-raising-minus}
\end{align}
\end{lemma}

\begin{proof}
The two product identities follow from the definition of \(B_{r,k}\) and the
binomial ratios
$$\frac{\binom{r-1}{j-1}}{\binom rj}=\frac{j}{r},\quad
\frac{\binom{r-1}{j-1}}{\binom r{j-1}}=\frac{r+1-j}{r}.$$
The two integral identities follow by differentiating the displayed sums, which telescope, and checking
the value at \(x=1\), respectively \(x=0\).
\end{proof}

The next lemma packages the rank-vector sums in the Bernstein basis.  Here and
in the rest of this section, \(H_w^{(m)}\) denotes the polynomial determined
by the recurrences \eqref{eq:H-empty}--\eqref{eq:H-minus}.

\begin{lemma}[Bernstein packaging of rank vectors]\label[lemma]{lem:rank-packaging}
For every \(m\ge1\), every \(r\ge0\), and every word \(w\in\{+,-\}^r\),
\[
r!\,H_w^{(m)}(x)=\sum_{k=1}^{r+1}D_w^{(m)}(k)\,B_{r,k}(x).
\]
\end{lemma}

\begin{proof}
We induct on the length of \(w\).  For \(w=\varnothing\), both sides equal
\(1\), since \(D_\varnothing^{(m)}(1)=1\) and \(B_{0,1}=1\).

Let \(w=+v\), where \(v\) has length \(r-1\ge0\), and assume the statement
for \(v\).  By \eqref{eq:H-plus} and the induction hypothesis,
\[
r!\,H_{+v}^{(m)}(x)
=r\sum_{j=1}^{r}D_v^{(m)}(j)
\left(xB_{r-1,j}(x)+m\int_x^1B_{r-1,j}(u)\dd u\right).
\]
By the degree-raising identities \eqref{eq:degree-raising-plus}, the
right-hand side equals
\[
\begin{aligned}
&\sum_{j=1}^{r}D_v^{(m)}(j)
\left(jB_{r,j+1}(x)+m\sum_{h=1}^{j}B_{r,h}(x)\right)\\
&=\sum_{k=1}^{r+1}
\left((k-1)D_v^{(m)}(k-1)+m\sum_{j=k}^{r}D_v^{(m)}(j)\right)B_{r,k}(x),
\end{aligned}
\]
where the coefficient of \(B_{r,k}\) collects the term \(j=k-1\) from the
first summand and the terms \(j\ge k\) from the second.  By
\eqref{eq:D-plus}, this coefficient is \(D_{+v}^{(m)}(k)\).

The case \(w=-v\) is identical, using \eqref{eq:H-minus},
\eqref{eq:degree-raising-minus}, and \eqref{eq:D-minus}.
\end{proof}

The following positive kernel converts endpoint counts into rank-vector sums.

\begin{definition}\label[definition]{def:kappa-kernel}
For \(r\ge m-1\), \(1\le k\le r+1\), and \(1\le\ell\le m\), set
\[
\kappa_r^{(m)}(k,\ell)
=(k-1)!(r+1-k)!
  \binom{m-1}{\ell-1}\binom{r-m+1}{k-\ell}.
\]
We use the convention that \(\binom ab=0\) when \(b<0\) or \(b>a\).
\end{definition}

\begin{lemma}[Degree elevation for the kernel]\label[lemma]{lem:kappa-degree-elevation}
If \(r\ge m-1\), then
\[
b_{m,\ell}(x)
=\sum_{k=1}^{r+1}\frac{\kappa_r^{(m)}(k,\ell)}{r!}\,B_{r,k}(x)
\]
for every \(1\le\ell\le m\).
\end{lemma}

\begin{proof}
This is the standard degree elevation of the Bernstein basis
\cite{Farouki}, written in the present normalization.  Substituting the definitions gives
\[
\frac{\kappa_r^{(m)}(k,\ell)}{r!}B_{r,k}(x)
=\binom{m-1}{\ell-1}\binom{r-m+1}{k-\ell}
  x^{k-1}(1-x)^{r+1-k}.
\]
After summing over \(k\), factor out
\(\binom{m-1}{\ell-1}x^{\ell-1}(1-x)^{m-\ell}\).  The remaining sum is
\[
\sum_{a=0}^{r-m+1}\binom{r-m+1}{a}x^a(1-x)^{r-m+1-a}=1
\]
by the binomial theorem.  The result is exactly \(b_{m,\ell}(x)\).
\end{proof}

Combining the packaging lemma with the transfer lemma and degree elevation
gives the finite transfer directly.

\begin{theorem}[Finite Bernstein transfer]\label{thm:finite-transfer-identity}
Let \(m\ge1\), and let \(w\in\{+,-\}^r\) with \(r\ge m-1\).  Then, for
\(1\le k\le r+1\),
\[
D_w^{(m)}(k)
=\sum_{\ell=1}^{m}C_w^{(m)}(\ell)\,\kappa_r^{(m)}(k,\ell).
\]
\end{theorem}

\begin{proof}
By \Cref{lem:transfer} and \Cref{lem:kappa-degree-elevation},
\[
r!\,H_w^{(m)}(x)
=r!\sum_{\ell=1}^{m}C_w^{(m)}(\ell)\,b_{m,\ell}(x)
=\sum_{k=1}^{r+1}
\left(\sum_{\ell=1}^{m}C_w^{(m)}(\ell)\,\kappa_r^{(m)}(k,\ell)\right)
B_{r,k}(x).
\]
By \Cref{lem:rank-packaging}, the left-hand side is also
\[
\sum_{k=1}^{r+1}D_w^{(m)}(k)B_{r,k}(x).
\]
The polynomials \(B_{r,1},\ldots,B_{r,r+1}\) form a basis of the space of
polynomials of degree at most \(r\), so the coefficients agree.
\end{proof}

\begin{corollary}[Diagonal case]\label[corollary]{cor:diagonal-case}
Let \(m\ge1\) and \(w\in\{+,-\}^{m-1}\).  Then
\[
D_w^{(m)}(k)=(m-1)!\,C_w^{(m)}(k)
\qquad(1\le k\le m).
\]
Equivalently, in the notation of \Cref{lem:finite-entry}: for \(n\ge1\),
\(\eps\in\{+,-\}^{n-1}\), \(w=\eps_{n-1}\cdots\eps_1\), and \(1\le k\le n\),
\[
\sum_{\substack{\pi\in S_n\\ \pi_n=k}}n^{\rec_\eps(\pi)}
=n!\,C_w^{(n)}(k).
\]
\end{corollary}

\begin{proof}
Take \(r=m-1\) in \Cref{thm:finite-transfer-identity}: the factor
\(\binom{r-m+1}{k-\ell}=\binom0{k-\ell}\) in \Cref{def:kappa-kernel} vanishes
unless \(k=\ell\), so \(\kappa_{m-1}^{(m)}(k,\ell)=(m-1)!\,\delta_{k,\ell}\).
The permutation form is the case \(m=n\), \(r=n-1\) of the first identity,
multiplied by \(n\) using \Cref{lem:finite-entry}.
\end{proof}

Integrating the degree-elevation identity in
\Cref{lem:kappa-degree-elevation} gives the column sum
\[
\sum_{k=1}^{r+1}\kappa_r^{(m)}(k,\ell)=\frac{(r+1)!}{m}.
\]
Here we used
\(\int_0^1B_{r,k}(x)\dd x=1/(r+1)\) and
\(\int_0^1b_{m,\ell}(x)\dd x=1/m\).
Consequently the finite transfer gives a second, probability-free proof of
\Cref{thm:main-general}.  Indeed, for
\(w=\eps_{n-1}\cdots\eps_1\) and \(1\le m\le n\),
\[
\begin{aligned}
\sum_{\pi\in S_n}m^{\rec_\eps(\pi)}
&=m\sum_{k=1}^{n}D_w^{(m)}(k)\\
&=m\sum_{\ell=1}^{m}C_w^{(m)}(\ell)
  \sum_{k=1}^{n}\kappa_{n-1}^{(m)}(k,\ell)\\
&=m\sum_{\ell=1}^{m}C_w^{(m)}(\ell)\frac{(n-1)!n}{m}\\
&=n!\,\Omega(P_\eps;m).
\end{aligned}
\]
The two degree-\(\le n\) polynomials also agree at \(m=0\), so interpolation
completes the argument.

\section{A bijective form of the finite transfer}
\label{sec:direct-bijection}

We construct an object-level bijection between pairs \((\sigma,f)\), with
\(\sigma\in S_n\) and \(f:P_\eps\to[m]\) order-preserving, and pairs
\((\pi,\lambda)\), with \(\pi\in S_n\) and
\(\lambda:\operatorname{Rec}_\eps(\pi)\to[m]\).  A weak map \(f\) has
\(\prod_{a\in[m]}|f^{-1}(a)|!\) linear refinements inside its level fibers,
so direct rank discretization is not fiberwise uniform.  The window kernel
records this ordering data.

The construction has three ingredients: a window model for the
degree-elevation kernel, a local window exchange, and a hole construction for
the diagonal identity
\(D_w^{(m)}(k)=(m-1)!\,C_w^{(m)}(k)\).
\Cref{thm:direct-bijection-m-le-n} gives a direct recursive bijection when
\(m\le n\); \Cref{thm:direct-bijection-all-m} extends it algorithmically to
arbitrary alphabets by a finite used-label sieve.

Throughout this section, if \(A\) is a finite totally ordered set and
\(a\in A\), we write
\[
\operatorname{std}_{A\setminus\{a\}}:A\setminus\{a\}\longrightarrow [|A|-1]
\]
for the increasing bijection.  When \(A=[N]\), this is the map
\[
x\longmapsto
\begin{cases}
x, & x<a,\\
x-1, & x>a.
\end{cases}
\]
We use the inverse map without further comment when we pass back from a
standardized set to the original set.

Throughout the constructions below, an instruction to insert an entry in the
\(b\)-th position of a word or board of length \(L-1\) means that the entry is
inserted after the first \(b-1\) entries, so that it occupies position \(b\) in
the resulting word or board of length \(L\).  Thus \(b=1\) means insertion at
the beginning and \(b=L\) means insertion at the end.

Let \(N\ge m\).  For \(1\le k\le N\) and \(1\le \ell\le m\), define
\(\mathcal K_{N,m}(k,\ell)\) to be the set of words $\omega=(k;r_1,r_2,\ldots,r_{N-1})$, whose entries are the elements of \([N]\), with first entry \(k\), such that
\(k\) has rank \(\ell\) in the \(m\)-window $k,r_1,\ldots,r_{m-1}$. Equivalently,
\[
\ell=1+\#\{1\le s\le m-1:r_s<k\}.
\]
We use the convention that \(\mathcal K_{N,m}(k,\ell)=\varnothing\) whenever
one of the indices is outside its natural range.

Direct counting gives
\[
|\mathcal K_{N,m}(k,\ell)|
=
(k-1)!(N-k)!
\binom{m-1}{\ell-1}\binom{N-m}{k-\ell}
=
\kappa_{N-1}^{(m)}(k,\ell).
\]
Indeed, choose the \(\ell-1\) window elements below \(k\) and the
\(m-\ell\) window elements above \(k\), then order the window and its
complement.  The count is
\[
\binom{k-1}{\ell-1}
\binom{N-k}{m-\ell}
(m-1)!(N-m)!,
\]
which simplifies to the displayed kernel.

The following local exchange realizes one step of the finite transfer.  First
suppose \(N>m\), so that the buffer entry \(r_m\) exists.

\begin{lemma}[The \(+\)-window exchange]
\label[lemma]{lem:plus-window-exchange}
Fix \(N>m\), \(1\le k\le N\), and \(1\le j\le m\).  There is an explicit
bijection
\[
\begin{aligned}
\bigsqcup_{\ell\le j}\mathcal K_{N,m}(k,\ell)
&\longleftrightarrow
\bigl(\{1,\ldots,k-1\}\times\mathcal K_{N-1,m}(k-1,j)\bigr)\\
&\qquad\sqcup
\left([m]\times\bigsqcup_{q=k}^{N-1}\mathcal K_{N-1,m}(q,j)\right).
\end{aligned}
\]
\end{lemma}

\begin{proof}
Take $\omega=(k;r_1,\ldots,r_{N-1})\in\mathcal K_{N,m}(k,\ell), \ell\le j$, and set
\[
\rho=1+\#\{1\le s\le m-1:r_s<k\}.
\]
Thus \(\rho=\ell\le j\).  Let $S=(r_1,\ldots,r_m)$ be the extended window.

If $\rho=j$ and $r_m<k$, we are in the non-record case.  Put \(a=r_m\).  Delete \(a\) from \(\omega\)
and standardize the remaining entries.  The threshold \(k\) becomes \(k-1\),
and the first \(m\)-window still has threshold rank \(j\).  Hence the output is
\[
(a,\omega')\in
\{1,\ldots,k-1\}\times\mathcal K_{N-1,m}(k-1,j).
\]

In all other cases, let \(a\) be the \(j\)-th smallest element of the extended
window \(S\), and let \(b\in[m]\) be its position in \(S\).  We claim that
\(a>k\).  Indeed, if \(r_m>k\), then the extended window has at most \(j-1\)
entries smaller than \(k\).  If \(r_m<k\), then failure of the non-record case
forces \(\rho<j\), and again the extended window has at most \(j-1\) entries
smaller than \(k\).  Thus the \(j\)-th smallest entry of \(S\) is larger than
\(k\).

This is the record case.  Delete the old threshold \(k\), put the new threshold
\(a\) in front, remove \(a\) from the rest word, and standardize.  Since
\(a>k\), the new threshold has standardized value $q=a-1, k\le q\le N-1$. The first \(m\)-window after this operation has the same underlying set as
\(S\), hence the new threshold has rank \(j\).  We output
\[
(b,\omega')\in
[m]\times \mathcal K_{N-1,m}(q,j).
\]

The inverse map is explicit.  In the non-record branch, start with
\[
a<k,\qquad \omega'\in\mathcal K_{N-1,m}(k-1,j).
\]
Unstandardize by inserting \(a\), so that the threshold \(k-1\) becomes \(k\),
and insert \(a\) as the \(m\)-th entry of the rest word.  This recovers the
unique \(\omega\) which falls into the non-record case.

In the record branch, start with
\[
b\in[m],\qquad
\omega'\in\mathcal K_{N-1,m}(q,j),\qquad k\le q\le N-1.
\]
Unstandardize by inserting the old threshold \(k\).  The current threshold
\(q\) becomes \(a=q+1\).  Remove this leading \(a\), put \(k\) in front, and
insert \(a\) in the \(b\)-th position of the rest word.  Its extended window
has \(a>k\) as its \(j\)-th smallest entry.  Every entry below \(k\) is below
\(a\), so the rank \(\ell\) of \(k\) in the reconstructed first window
satisfies \(\ell\le j\); hence the object lies in the domain on the left.
It cannot satisfy simultaneously \(\ell=j\) and \(r_m<k\), because then the
extended window would contain at least \(j\) entries below \(k<a\), contrary
to \(a\) being its \(j\)-th smallest entry.  Thus the reconstruction lies in
the record branch, with the prescribed position \(b\).  The two inverse
constructions are therefore well-defined and inverse to the forward branches.
\end{proof}

\begin{corollary}[The \(-\)-window exchange]
\label[corollary]{lem:minus-window-exchange}
Fix \(N>m\), \(1\le k\le N\), and \(1\le j\le m\).  There is an explicit
bijection
\[
\begin{aligned}
\bigsqcup_{\ell\ge j}\mathcal K_{N,m}(k,\ell)
&\longleftrightarrow
\bigl(\{k+1,\ldots,N\}\times\mathcal K_{N-1,m}(k,j)\bigr)\\
&\qquad\sqcup
\left([m]\times\bigsqcup_{q=1}^{k-1}\mathcal K_{N-1,m}(q,j)\right).
\end{aligned}
\]
\end{corollary}

\begin{proof}
Apply the order-reversing involution
\[
(k;r_1,\ldots,r_{N-1})
\longmapsto
(N+1-k;N+1-r_1,\ldots,N+1-r_{N-1}).
\]
It sends \(\mathcal K_{N,m}(k,\ell)\) to
\(\mathcal K_{N,m}(N+1-k,m+1-\ell)\).  Therefore \(\ell\ge j\) becomes
\(m+1-\ell\le m+1-j\), and
\Cref{lem:plus-window-exchange}, with threshold \(N+1-k\) and index
\(m+1-j\), applies.  On the size-\(N-1\) output word use the corresponding
order reversal \(x\mapsto N-x\).  Its non-record branch becomes
\[
\{k+1,\ldots,N\}\times\mathcal K_{N-1,m}(k,j),
\]
and its record thresholds \(q'=N+1-k,\ldots,N-1\) become
\(q=N-q'=1,\ldots,k-1\).  Positions in the extended window are unchanged,
so the label in \([m]\) is unchanged.  Conjugating the forward and inverse
maps of \Cref{lem:plus-window-exchange} by these order reversals proves the
stated bijection and its reversibility.
\end{proof}

When \(N=m\), the buffer entry \(r_m\) is absent; the diagonal case is handled
by the following hole construction.

\begin{lemma}[The diagonal hole bijection]
\label[lemma]{lem:diagonal-hole-bijection}
Let \(m\ge1\), let \(w\in\{+,-\}^{m-1}\), and fix \(1\le k\le m\).  There is a
bijection
\[
\begin{aligned}
&\{(\alpha,\lambda):
  \alpha\text{ is a scan order of }[m]\setminus\{k\},\
  \lambda:\operatorname{Rec}_w(\alpha;k)\to[m]\}\\
\longleftrightarrow
&\{(g,\rho):g\text{ is counted by }C_w^{(m)}(k),\ \rho\in S_{m-1}\}.
\end{aligned}
\]
Here \(\operatorname{Rec}_w(\alpha;k)\) denotes the set of record steps in the
threshold scan with initial threshold \(k\) and comparison word \(w\).
\end{lemma}

\begin{proof}
We give the map from labeled scan orders to pairs \((g,\rho)\).  Maintain a
board with \(m\) positions.  Each position contains either an actual value
\(A_x\), \(x\in[m]\), or a hole \(H_t\), where \(t\) records the time at which
the hole was created.  Initially the board is $(A_1,A_2,\ldots,A_m)$, and the threshold is \(A_k\). Thus \(g_0=k\).

Process the scan order $\alpha=(a_1,\ldots,a_{m-1})$ from left to right.  Suppose we are at time \(t\).  If the \(t\)-th inspected
value is not a record, replace the board entry \(A_{a_t}\) by the hole
\(H_t\), and keep the threshold unchanged.  If it is a record, let
\(b=\lambda(t)\).  Delete the old threshold from the board, insert the hole
\(H_t\) in the \(b\)-th position, and make \(A_{a_t}\) the new threshold.
After either operation, define \(g_t\) to be the current position of the
threshold.

After time \(t\), the following board invariant holds:
\begin{enumerate}[label=(\roman*)]
\item the holes are exactly \(H_1,\ldots,H_t\);
\item the actual tokens are the current threshold together with the
uninspected values \(A_{a_{t+1}},\ldots,A_{a_{m-1}}\), and they occur in
increasing value order when the holes are ignored;
\item the current threshold occupies position \(g_t\).
\end{enumerate}
This follows by induction: a non-record replaces its inspected token by a
hole, while a record deletes the old threshold and retains the inspected token
as the new threshold; neither operation changes the relative order of the
remaining actual tokens.  Consequently, at a \(+\)-step a non-record lies to
the left of the threshold and leaves its position unchanged, whereas a record
selects an actual token to its right, so the threshold position weakly
increases.  Thus
\[
w_t=+\Longrightarrow g_t\ge g_{t-1}.
\]
The same argument with left and right interchanged gives
\[
w_t=-\Longrightarrow g_t\le g_{t-1}.
\]
After \(m-1\) steps, only the final threshold remains actual; the other
entries are the holes \(H_1,\ldots,H_{m-1}\).  Reading their time labels from
left to right gives a permutation \(\rho\in S_{m-1}\).

For the inverse map, start from \((g,\rho)\) and construct the final
board by putting one unnamed actual token in position \(g_{m-1}\), and by
putting the holes
\[
H_{\rho_1},H_{\rho_2},\ldots,H_{\rho_{m-1}}
\]
from left to right in the other positions.  This actual token is the current
threshold.

For \(t=m-1,m-2,\ldots,1\), let \(q=g_t\) be the current threshold position
and let \(p\) be the current position of \(H_t\).  If
\[
w_t=+,\qquad g_t=g_{t-1},\qquad p<q,
\]
or if
\[
w_t=-,\qquad g_t=g_{t-1},\qquad p>q,
\]
then the original step was a non-record: replace \(H_t\) by a new actual
token, and declare this token to be the \(t\)-th scanned token.  Otherwise the
original step was a record: the current threshold token is the \(t\)-th scanned
token, the label is $\lambda(t)=p$, then delete \(H_t\), insert a new actual token in position \(g_{t-1}\), and
make this new token the threshold.

After reversing all steps, the board contains \(m\) actual tokens and no holes.
Assign the values \(1,2,\ldots,m\) to these tokens from left to right.  Since
the threshold is in position \(g_0=k\), the initial threshold receives value
\(k\).  The values assigned to the scanned tokens, in times
\(1,\ldots,m-1\), form \(\alpha\), and the labels recovered in the record
steps form \(\lambda\).

To prove reversibility, use descending induction on \(t\).  After the steps
\(m-1,m-2,\ldots,t+1\) have been undone, the board is, up to the still unnamed
actual tokens, exactly the forward board just after time \(t\), with the same
threshold and holes.  For a \(+\)-step, a forward non-record creates \(H_t\)
strictly to the left of the unchanged threshold, so
\(g_t=g_{t-1}\) and \(p<q\).  In a forward record the new threshold comes from
the right of the old one.  If deletion of the old threshold and insertion of
\(H_t\) leave the threshold in position \(g_{t-1}\), then the new threshold
was the next actual token to the right and the inserted hole lies to its right;
thus \(p<q\) is impossible.  Hence the inverse criterion detects exactly the
non-record \(+\)-steps.  The argument for a \(-\)-step is obtained by
interchanging left and right.  In either branch the stated reverse operation
is the unique inverse of the forward operation, so the induction invariant is
restored at time \(t-1\).  After all steps are undone, assigning values in
left-to-right order recovers the comparisons, the scan order, and every record
label.  The constructions are mutually inverse.
\end{proof}

\begin{example}[The hole construction in action]
\label[example]{ex:hole-bijection}
Take \(m=4\), \(w=(+,+,-)\), initial threshold \(k=1\), scan order
\(\alpha=(4,3,2)\), and labels \(\lambda(1)=2\) and \(\lambda(3)=1\).  The board
starts as \([A_1,A_2,A_3,A_4]\), with the threshold \(A_1\) in position
\(g_0=1\).
\begin{itemize}
\item \(t=1\) (\(+\)): \(a_1=4>1\) is a record.  Delete the threshold \(A_1\) and
  insert the hole \(H_1\) in position \(\lambda(1)=2\); the new threshold is
  \(A_4\).  The board is \([A_2,H_1,A_3,A_4]\) and \(g_1=4\).
\item \(t=2\) (\(+\)): \(a_2=3<4\) is a non-record.  Replace \(A_3\) by \(H_2\)
  in place, leaving the threshold \(A_4\) unchanged.  The board is
  \([A_2,H_1,H_2,A_4]\) and \(g_2=4\).
\item \(t=3\) (\(-\)): \(a_3=2<4\) is a record.  Delete the threshold \(A_4\) and
  insert the hole \(H_3\) in position \(\lambda(3)=1\); the new threshold is
  \(A_2\).  The board is \([H_3,A_2,H_1,H_2]\) and \(g_3=2\).
\end{itemize}
Reading the hole time-labels from left to right gives \(\rho=(3,1,2)\), and the
endpoint sequence is \(g=(1,4,4,2)\), weakly increasing across the two
\(+\)-steps and weakly decreasing across the \(-\)-step, as required of a path
counted by \(C_w^{(4)}(1)\).  Thus \((\alpha,\lambda)\mapsto(g,\rho)\); the
inverse rebuilds the board from \(g\) and \(\rho\), reversing the three steps to
recover \((\alpha,\lambda)\).
\end{example}

We can now assemble the finite-transfer bijection.  For \(N\ge m\), \(w\in
\{+,-\}^{N-1}\), and \(1\le k\le N\), let \(\mathcal R_{N,m}(w;k)\) denote
the set of pairs \((\alpha,\lambda)\), where \(\alpha\) is a scan order of
\([N]\setminus\{k\}\) and
\[
\lambda:\operatorname{Rec}_w(\alpha;k)\to[m].
\]
Let \(\mathcal T_{N,m}(w;k)\) be the disjoint union
\[
\mathcal T_{N,m}(w;k)
=
\bigsqcup_{\ell=1}^{m}
\{g:g\text{ is counted by }C_w^{(m)}(\ell)\}
\times
\mathcal K_{N,m}(k,\ell).
\]

\begin{proposition}[Finite-transfer bijection]
\label[proposition]{prop:finite-transfer-bijection}
For \(N\ge m\), \(w\in\{+,-\}^{N-1}\), and \(1\le k\le N\), there is a
bijection
\[
\Theta_{N,m,w,k}:\mathcal T_{N,m}(w;k)
\longrightarrow
\mathcal R_{N,m}(w;k).
\]
\end{proposition}

\begin{proof}
We define \(\Theta_{N,m,w,k}\) recursively on \(N\).  If \(N=m\), then
\(\mathcal K_{m,m}(k,\ell)\) is empty unless \(\ell=k\).  For $\omega=(k;r_1,\ldots,r_{m-1})\in\mathcal K_{m,m}(k,k)$, standardize the rest word by deleting \(k\):
\[
\rho=
\bigl(\operatorname{std}_{[m]\setminus\{k\}}(r_1),
\ldots,
\operatorname{std}_{[m]\setminus\{k\}}(r_{m-1})\bigr)
\in S_{m-1}.
\]
Then apply the inverse direction of \Cref{lem:diagonal-hole-bijection} to
\((g,\rho)\).

Assume now that \(N>m\), and write \(w=sv\), where \(s\in\{+,-\}\).  Take an
input \((g,\omega)\in\mathcal T_{N,m}(w;k)\), with $g=(g_0,g_1,\ldots,g_{N-1}), g_0=\ell$. Put \(j=g_1\) and \(g'=(g_1,\ldots,g_{N-1})\).  In every recursive call below,
the labeled scan order returned by the recursive bijection is a tail scan of
length \(N-2\).  When this tail is attached after the first scanned value, its
scan times are reindexed by \(u\mapsto u+1\), and the labels on its record
steps are transported by this reindexing.

If \(s=+\), then
\(\ell\le j\), so \Cref{lem:plus-window-exchange} applies to \(\omega\).
There are two cases.

In the non-record case it returns $a<k, \omega'\in\mathcal K_{N-1,m}(k-1,j)$. Recursively apply \(\Theta_{N-1,m,v,k-1}\) to \((g',\omega')\), obtaining a labeled scan order on the standardized set. Unstandardize the tail by
reinserting \(a\), and put \(a\) as the first scanned value.  No new label is
added.

In the record case it returns
\[
b\in[m],\qquad
\omega'\in\mathcal K_{N-1,m}(q,j),
\qquad k\le q\le N-1.
\]
The first scanned value is the unstandardized value \(a=q+1\).  Recursively
apply \(\Theta_{N-1,m,v,q}\) to \((g',\omega')\), unstandardize the tail by
reinserting the old threshold \(k\), put \(a\) first, and give this first
record the label \(b\).

The construction for \(s=-\) is identical, using
\Cref{lem:minus-window-exchange}.  In the non-record case one has $a>k, \omega'\in\mathcal K_{N-1,m}(k,j)$. Recursively apply \(\Theta_{N-1,m,v,k}\) to \((g',\omega')\), unstandardize the
tail by reinserting \(a\), and put \(a\) first, with no new label.  In the
record case one has
\[
b\in[m],\qquad
\omega'\in\mathcal K_{N-1,m}(q,j),
\qquad 1\le q\le k-1.
\]
The first scanned value is the unstandardized value \(a=q<k\).  Recursively
apply \(\Theta_{N-1,m,v,q}\) to \((g',\omega')\), unstandardize the tail by
reinserting the old threshold \(k\), put \(a\) first, and give this first
record the label \(b\).

The inverse recursion is obtained by reversing these steps.  If \(N=m\), apply
the forward direction of \Cref{lem:diagonal-hole-bijection} to the labeled
scan order, obtaining \((g,\rho)\), and then unstandardize \(\rho\) by
reinserting \(k\) to recover $\omega=(k;r_1,\ldots,r_{m-1})$.
If \(N>m\), read the first scanned value \(a\).  Remove this first scanned
value from the scan order, and reindex every remaining scan time \(i\) as
\(i-1\); the labels on record steps in the tail are transported by the same
reindexing.  For a \(+\)-step, \(a<k\) is the non-record case and \(a>k\) is
the record case; for a \(-\)-step, \(a>k\) is the non-record case and \(a<k\)
is the record case.

In a non-record case there is no label at the first scan time.  Standardize the
tail by deleting \(a\); its initial threshold is \(k-1\) for a \(+\)-step and
\(k\) for a \(-\)-step.  Apply the inverse recursion to recover \(g'\) and
\(\omega'\), and then use the inverse non-record branch of the appropriate
local exchange.  In a record case the first scan time has a unique label
\(b\).  Remove it, standardize the tail by deleting the old threshold \(k\),
and use initial threshold \(q=a-1\) for a \(+\)-step and \(q=a\) for a
\(-\)-step.  The inverse recursion followed by the inverse record branch of
\Cref{lem:plus-window-exchange} or \Cref{lem:minus-window-exchange} recovers
\(\omega\).

In both cases let \(\ell\) be the rank of \(k\) in the reconstructed first
\(m\)-window and set \(g=(\ell,g')\).  The inverse local exchange places its
input in a summand with \(\ell\le j=g'_0\) when \(s=+\), and with
\(\ell\ge j\) when \(s=-\).  Since, inductively, \(g'\) is counted by
\(C_v^{(m)}(j)\), these inequalities show that \(g\) is counted by
\(C_{sv}^{(m)}(\ell)\).  Thus the reconstructed pair lies in
\(\mathcal T_{N,m}(w;k)\), and every parameter used by the forward recursion
is recovered, so the inverse recursion proves bijectivity.
\end{proof}

\begin{theorem}[An explicit bijection for \(m\le n\)]
\label{thm:direct-bijection-m-le-n}
Let \(1\le m\le n\).  For every
\(\eps\in\{+,-\}^{n-1}\), there is an explicit bijection between pairs
\[
(\sigma,f),\qquad
\sigma\in S_n,\quad f:P_\eps\to[m]\text{ order-preserving},
\]
and pairs
\[
(\pi,\lambda),\qquad
\pi\in S_n,\quad
\lambda:\operatorname{Rec}_\eps(\pi)\to[m].
\]
\end{theorem}

\begin{proof}
Put $w=\eps_{n-1}\eps_{n-2}\cdots\eps_1$. Given \((\sigma,f)\), set $g_i=f(x_{n-i}), 0\le i\le n-1$. Then \(g\) is counted by \(C_w^{(m)}(g_0)\).  Let \(\ell=g_0\).  In the first
\(m\) entries of \(\sigma\), let \(k\) be the \(\ell\)-th smallest entry, and
let \(a_0\in[m]\) be its position.  Move \(k\) to the first position of
\(\sigma\), preserving the relative order of all other entries.  Because \(k\)
starts among the first \(m\) entries, this operation preserves the underlying
set of the first \(m\) entries; in the resulting word \(k\) still has rank
\(\ell\) in that set.  Thus it gives $\omega\in\mathcal K_{n,m}(k,\ell)$. Apply $\Theta_{n,m,w,k}(g,\omega)$ to obtain a scan order $\alpha=(\alpha_1,\ldots,\alpha_{n-1})$
of \([n]\setminus\{k\}\), with labels on its non-terminal records.  Define
\[
\pi_n=k,\qquad
\pi_{n-i}=\alpha_i\quad1\le i\le n-1,
\]
and give the terminal record \(n\) the label $\lambda(n)=a_0$.
The remaining labels are those produced by \(\Theta_{n,m,w,k}\), translated
from scan time \(i\) to position \(n-i\).  This produces a labeled greedy
\(\eps\)-record object \((\pi,\lambda)\).

The inverse starts with \((\pi,\lambda)\).  Put \(k=\pi_n\) and
\(a_0=\lambda(n)\), and form the scan order $\alpha=(\pi_{n-1},\pi_{n-2},\ldots,\pi_1)$. Define a labeling \(\lambda_{\mathrm{tail}}\) of the non-terminal record steps
of this scan by
\[
\lambda_{\mathrm{tail}}(i)=\lambda(n-i)
\qquad
i\in \operatorname{Rec}_w(\alpha;k).
\]
Apply \(\Theta_{n,m,w,k}^{-1}\) to
\((\alpha,\lambda_{\mathrm{tail}})\).  This recovers \(g\) and $\omega=(k;r_1,\ldots,r_{n-1})\in\mathcal K_{n,m}(k,g_0)$.

Finally, insert the leading \(k\) back into the \(a_0\)-th position of the
first \(m\)-window:
\[
\sigma=(r_1,\ldots,r_{a_0-1},k,r_{a_0},r_{a_0+1},\ldots,r_{n-1}).
\]
Since \(a_0\le m\), this is exactly the inverse of moving \(k\) to the front:
it restores the first-window set and the relative order of every other entry.
Set
\[
f(x_{n-i})=g_i,\qquad 0\le i\le n-1.
\]
The inequalities defining \(C_w^{(m)}\) are exactly the order-preserving
conditions for \(P_\eps\), so this recovers a unique pair \((\sigma,f)\).
All steps are inverse to the forward construction.
\end{proof}

To pass from \(m\le n\) to arbitrary \(m\), observe that every map
\(f:P_\eps\to[m]\), and every record labeling
\(\lambda:\operatorname{Rec}_\eps(\pi)\to[m]\), uses at most \(n\) labels.

For a nonempty subset \(T\subseteq[m]\), let \(\mathcal A_T\) be the set of
pairs \((\sigma,f)\) with \(f:P_\eps\to T\) order-preserving, where \(T\) has
the inherited order.  Let \(\mathcal B_T\) be the set of pairs
\((\pi,\lambda)\) with
\[
\lambda:\operatorname{Rec}_\eps(\pi)\to T.
\]
If \(1\le |T|=t\le n\), let \(c_T:[t]\to T\) be the increasing bijection.  The
bijection of \Cref{thm:direct-bijection-m-le-n}, transported by \(c_T\), gives
a bijection
\[
\Phi_T:\mathcal A_T\longrightarrow \mathcal B_T.
\]
We shall use \(\Phi_T\) only for subsets \(T\) of size at most \(n\).

\begin{corollary}[Arbitrary alphabets]
\label{thm:direct-bijection-all-m}
For every \(m,n\ge1\) and every \(\eps\in\{+,-\}^{n-1}\), there is an
algorithmic bijection between pairs
\[
(\sigma,f),\qquad
\sigma\in S_n,\quad f:P_\eps\to[m]\text{ order-preserving},
\]
and pairs
\[
(\pi,\lambda),\qquad
\pi\in S_n,\quad
\lambda:\operatorname{Rec}_\eps(\pi)\to[m].
\]
\end{corollary}

\begin{proof}
Fix the nonempty set \(S\subseteq[m]\) of labels used by an object.  Both
strata are empty when \(|S|>n\), so assume \(|S|\le n\), and write
\[
\mathcal A_S^{=}
=\{(\sigma,f)\in\mathcal A_S:\operatorname{im}(f)=S\},
\qquad
\mathcal B_S^{=}
=\{(\pi,\lambda)\in\mathcal B_S:\operatorname{im}(\lambda)=S\}.
\]
Apply the Garsia--Milne involution principle \cite{GarsiaMilne} to
\[
\widetilde{\mathcal A}_S
=\bigsqcup_{\varnothing\ne T\subseteq S}\{T\}\times\mathcal A_T,
\qquad
\widetilde{\mathcal B}_S
=\bigsqcup_{\varnothing\ne T\subseteq S}\{T\}\times\mathcal B_T,
\]
with sign \((-1)^{|S|-|T|}\) on the \(T\)-summand.  On either side, unless
all labels of \(S\) are used, toggle
\(c=\min(S\setminus\operatorname{im})\) in \(T\).  This is a sign-reversing
involution whose fixed points are, respectively,
\(\{S\}\times\mathcal A_S^{=}\) and
\(\{S\}\times\mathcal B_S^{=}\).  Since \(|T|\le n\), the maps \(\Phi_T\)
assemble into a sign-preserving bijection between the two signed sets.
The involution principle therefore gives an algorithmic bijection
\(\mathcal A_S^{=}\longrightarrow\mathcal B_S^{=}\).  Taking the disjoint
union over the unique used-label set \(S\) proves the result.
\end{proof}

\section{Directional and record-set refinements}
\label{sec:signed-strict-edge-transfer}

\subsection{Decorated directional transfer}

The continuous process separates \(+\)-records from \(-\)-records.  For
indeterminates \(p,q\) and a word \(w\in\{+,-\}^r\), define
\(K_w^{(p,q)}(x)\in\mathbb Q[p,q][x]\) by
\begin{align*}
K_\varnothing^{(p,q)}(x)&=1,\\
K_{+v}^{(p,q)}(x)
&=xK_v^{(p,q)}(x)+p\int_x^1K_v^{(p,q)}(u)\dd u,\\
K_{-v}^{(p,q)}(x)
&=(1-x)K_v^{(p,q)}(x)+q\int_0^xK_v^{(p,q)}(u)\dd u.
\end{align*}
The same conditioning argument as in the proof of
\Cref{thm:main-general} gives, for
\(w=\eps_{n-1}\eps_{n-2}\cdots\eps_1\),
\[
\sum_{\pi\in S_n}
p^{\rec_\eps^+(\pi)}q^{\rec_\eps^-(\pi)}
=
n!\int_0^1K_w^{(p,q)}(x)\dd x.
\]
Moreover \(K_w^{(m,m)}=H_w^{(m)}\).  At the finite level, the threshold
history alone forgets the value inspected at a non-record step.  We retain
that value as a decoration on the corresponding equality step.

Let \(w=w_1\cdots w_r\in\{+,-\}^r\), and let \(1\le k\le r+1\).  A
\emph{decorated endpoint path} of type \((w,k)\) is a pair \((g,\theta)\)
with
\[
g=(g_0,g_1,\ldots,g_r)\in[r+1]^{r+1},
\qquad g_0=k,
\]
such that
\[
w_i=+\Longrightarrow g_i\ge g_{i-1},
\qquad
w_i=-\Longrightarrow g_i\le g_{i-1}
\]
for \(1\le i\le r\), together with decorations on the equality steps.  More
precisely, set
\[
\begin{aligned}
S_+(g)&=\{i:w_i=+\text{ and }g_i>g_{i-1}\},\\
S_-(g)&=\{i:w_i=-\text{ and }g_i<g_{i-1}\},
\end{aligned}
\]
and
\[
E(g)=\{i:g_i=g_{i-1}\}.
\]
For every \(i\in E(g)\) we choose a value \(\theta_i\in[r+1]\), subject to
the following two conditions:
\[
w_i=+\Longrightarrow \theta_i<g_{i-1},
\qquad
w_i=-\Longrightarrow \theta_i>g_{i-1},
\]
and the list
\[
g_0,\qquad
g_i\ (i\in S_+(g)\cup S_-(g)),\qquad
\theta_i\ (i\in E(g))
\]
is a rearrangement of \([r+1]\).  Let \(\mathcal M_w(k)\) denote the set of
decorated endpoint paths of type \((w,k)\), and define
\[
R_w^{(p,q)}(k)
=
\sum_{(g,\theta)\in\mathcal M_w(k)}
p^{|S_+(g)|}q^{|S_-(g)|}.
\]
The scan-order bijection below gives
\[
R_w^{(m,m)}(k)=D_w^{(m)}(k),
\]
so this construction refines the rank-vector sums of
\Cref{def:rank-vector-D} by record direction.
The forgetful map \((g,\theta)\mapsto g\) preserves the strict \(+\)- and
strict \(-\)-steps, but its fibers need not be singletons.  For example, when
\(w=++\) and \(k=3\), the weak path \((3,3,3)\) has two admissible
decorations, corresponding to the two orders of the failed values \(1\) and
\(2\).  The decorations retain exactly this missing ordering data.

\begin{theorem}[Decorated signed transfer]
\label{thm:decorated-signed-transfer}
Let \(w\in\{+,-\}^r\).  Then
\[
r!\,K_w^{(p,q)}(x)
=
\sum_{k=1}^{r+1}R_w^{(p,q)}(k)B_{r,k}(x).
\]
Consequently, if \(n=r+1\) and
\(w=\eps_{n-1}\eps_{n-2}\cdots\eps_1\), then
\[
\sum_{\pi\in S_n}
p^{\rec_\eps^+(\pi)}q^{\rec_\eps^-(\pi)}
=
\sum_{k=1}^{n}R_w^{(p,q)}(k)
=
n!\int_0^1K_w^{(p,q)}(x)\dd x.
\]
In particular,
\[
n!\,\Omega(P_\eps;t)
=
t\sum_{k=1}^{n}R_w^{(t,t)}(k),
\]
where the factor \(t\) marks the terminal record.  Thus the
direction-refined record weights are matched, in this diagonal model, by
direction-marked strict edges on admissibly decorated order-preserving maps.
\end{theorem}

\begin{proof}
We first give the finite bijection behind \(R_w^{(p,q)}(k)\).  Fix
\(k\in[r+1]\), and let $\alpha=(a_1,\ldots,a_r)$ be a permutation of \([r+1]\setminus\{k\}\).  Start with threshold
\(h_0=k\).  At step \(i\), declare a record if
\[
w_i=+\text{ and }a_i>h_{i-1},
\qquad\text{or}\qquad
w_i=-\text{ and }a_i<h_{i-1}.
\]
If a record occurs, set \(h_i=a_i\); otherwise set \(h_i=h_{i-1}\).

From \(\alpha\) construct \((g,\theta)\) by setting \(g_i=h_i\).  If step
\(i\) is not a record, then \(g_i=g_{i-1}\) and we set $\theta_i=a_i$. If step \(i\) is a record, no decoration is added.  A record in a \(+\)-step
is exactly a strict \(+\)-step \(g_i>g_{i-1}\), and a record in a \(-\)-step
is exactly a strict \(-\)-step \(g_i<g_{i-1}\).  A non-record in a \(+\)-step
has \(a_i<h_{i-1}=g_{i-1}\), and a non-record in a \(-\)-step has
\(a_i>h_{i-1}=g_{i-1}\).  Since the entries \(k,a_1,\ldots,a_r\) form a
permutation of \([r+1]\), the defining list for \((g,\theta)\) is also a
permutation of \([r+1]\).  Hence \((g,\theta)\in\mathcal M_w(k)\).

Conversely, given \((g,\theta)\in\mathcal M_w(k)\), define
\[
a_i=
\begin{cases}
g_i, & \text{if }g_i\ne g_{i-1},\\
\theta_i, & \text{if }g_i=g_{i-1}.
\end{cases}
\]
The permutation condition for \(\mathcal M_w(k)\) implies that
\((a_1,\ldots,a_r)\) is a permutation of \([r+1]\setminus\{k\}\).
Moreover, the decoration inequalities force equality steps to be precisely
non-record steps, while strict steps are precisely records.  Thus the two
constructions are inverse to each other and preserve the weight
\(p^{|S_+|}q^{|S_-|}\).

This bijection gives the recurrences $R_\varnothing^{(p,q)}(1)=1$ and for \(v\in\{+,-\}^{r-1}\),
\begin{align}
R_{+v}^{(p,q)}(k)
&=(k-1)R_v^{(p,q)}(k-1)
  +p\sum_{j=k}^{r}R_v^{(p,q)}(j),
\label{eq:R-plus-signed}\\
R_{-v}^{(p,q)}(k)
&=(r+1-k)R_v^{(p,q)}(k)
  +q\sum_{j=1}^{k-1}R_v^{(p,q)}(j),
\label{eq:R-minus-signed}
\end{align}
with the boundary convention \(R_v^{(p,q)}(0)=R_v^{(p,q)}(r+1)=0\), so that the
out-of-range terms \((k-1)R_v^{(p,q)}(k-1)\) at \(k=1\) and
\((r+1-k)R_v^{(p,q)}(k)\) at \(k=r+1\) vanish.  For instance, in the
\(+\)-case, a first inspected value below \(k\) is a
non-record; after deleting it and standardizing, the threshold has rank
\(k-1\), giving the first term.  A first inspected value above \(k\) is a
record and contributes the factor \(p\); after deleting the old threshold and
standardizing, the new threshold has rank \(j\), where \(k\le j\le r\).  The
\(-\)-case is analogous.

The Bernstein expansion now follows by induction on \(r=|w|\).  Substitute
the expansion for the suffix \(v\) into the defining recurrence for
\(K_{+v}^{(p,q)}\) or \(K_{-v}^{(p,q)}\), and apply the four degree-raising
identities in \Cref{lem:degree-raising}.  The coefficient of \(B_{r,k}\) is
respectively the right-hand side of \eqref{eq:R-plus-signed} or
\eqref{eq:R-minus-signed}.  This is the same coefficient comparison as in
\Cref{lem:rank-packaging}, with the two record directions carrying separate
weights \(p\) and \(q\).

Finally, \(\int_0^1B_{r,k}(x)\dd x=1/(r+1)\) by the same beta integral as for
\(b_{m,k}\) above.  Therefore
\[
(r+1)!\int_0^1K_w^{(p,q)}(x)\dd x
=
\sum_{k=1}^{r+1}R_w^{(p,q)}(k).
\]
When \(w=\eps_{n-1}\cdots\eps_1\), reading a permutation
\(\pi\in S_n\) with terminal value \(\pi_n=k\) from right to left gives a
permutation of \([n]\setminus\{k\}\) with comparison word \(w\).  The bijection
above identifies its \(+\)-records and \(-\)-records with strict \(+\)- and
strict \(-\)-steps of the decorated endpoint path.  Summing over \(k\) gives
the direction-refined identity.  The order-polynomial specialization follows
from \Cref{thm:main-general}, since the terminal record contributes one additional
factor \(t\).
\end{proof}

\subsection{Record-set fibers and record posets}
\label{sec:record-set-fibers}

We refine the greedy statistic by its full record set.  For \(\pi\in S_n\),
recall that \(\operatorname{Rec}_\eps(\pi)\)
is the set of greedy \(\eps\)-record positions of \(\pi\).  Thus \(n\in
\operatorname{Rec}_\eps(\pi)\) and \(\rec_\eps(\pi)=
|\operatorname{Rec}_\eps(\pi)|\).

Fix \(R\subseteq[n]\) with \(n\in R\).  For \(i<n\), set
\[
\rho_R(i)=\min\{r\in R:r>i\}.
\]
Thus \(\rho_R(i)\) is the nearest element of \(R\) to the right of \(i\).

\begin{definition}\label[definition]{def:record-tree}
The record poset \(Q_{\eps,R}\) is the labeled poset on \([n]\) obtained as
follows.  For each \(i<n\), impose the relation
\[
 i<_{Q_{\eps,R}}\rho_R(i)
\]
if
\[
(\eps_i=+\text{ and }i\notin R)
\quad\text{or}\quad
(\eps_i=-\text{ and }i\in R),
\]
and otherwise impose the opposite relation
\[
 \rho_R(i)<_{Q_{\eps,R}} i.
\]
Then take the transitive closure.
\end{definition}

This indeed defines a poset.  The undirected graph formed by the defining
edges has one edge \(\{i,\rho_R(i)\}\) for each \(i<n\).  Iterating
\(i\mapsto\rho_R(i)\) always reaches \(n\), so the graph is connected; since
it has \(n-1\) edges, it is a tree.  Hence no orientation of these edges can
contain a directed cycle.

\begin{example}\label[example]{ex:record-tree}
Let \(\eps=(+,-,+)\) and \(\pi=2413\).  The terminal value is \(3\), and none
of the values \(1,4,2\), inspected from right to left, resets the threshold.
Thus \(\operatorname{Rec}_\eps(\pi)=\{4\}\).  Since
\(\rho_{\{4\}}(i)=4\) for \(i=1,2,3\), the defining relations of the record
tree are
\[
1<_{Q_{\eps,\{4\}}}4,
\qquad
4<_{Q_{\eps,\{4\}}}2,
\qquad
3<_{Q_{\eps,\{4\}}}4.
\]
The inverse word is \(\pi^{-1}=3142\), which respects all three relations.
\end{example}

\begin{lemma}\label[lemma]{lem:record-set-linear-extensions}
For every \(\pi\in S_n\) and every \(R\subseteq[n]\) with \(n\in R\),
\[
\operatorname{Rec}_\eps(\pi)=R
\quad\Longleftrightarrow\quad
\pi^{-1}\in\mathcal L(Q_{\eps,R}),
\]
where \(\mathcal L(Q_{\eps,R})\) denotes the set of linear extensions of
\(Q_{\eps,R}\).
\end{lemma}

\begin{proof}
Assume first that \(\operatorname{Rec}_\eps(\pi)=R\).  When the
right-to-left scan reaches a position \(i<n\), the current threshold is
\(\pi_{\rho_R(i)}\), because \(\rho_R(i)\) is the closest record position to
the right of \(i\).  Since the entries of \(\pi\) are distinct, the condition
that \(i\) is, or is not, declared a record is equivalent to
\[
\pi_i<\pi_{\rho_R(i)}
\quad\Longleftrightarrow\quad
(\eps_i=+\text{ and }i\notin R)
\quad\text{or}\quad
(\eps_i=-\text{ and }i\in R).
\]
This is exactly the comparison encoded by the defining relation between
\(i\) and \(\rho_R(i)\) in \(Q_{\eps,R}\).

Now \(\pi^{-1}\), written in one-line notation, is the word of positions of
\(1,2,\ldots,n\) in \(\pi\).  Therefore a position \(a\) occurs before a
position \(b\) in \(\pi^{-1}\) if and only if \(\pi_a<\pi_b\).  Hence the
comparisons above say precisely that \(\pi^{-1}\) respects every defining
relation of \(Q_{\eps,R}\), and hence is a linear extension.

Conversely, suppose \(\pi^{-1}\in\mathcal L(Q_{\eps,R})\).  We show by
downward induction on \(i\) that the greedy scan of \(\pi\) has record set
exactly \(R\).  The terminal position \(n\) lies in \(R\) and is always a
record.  Assume that every position \(>i\) has been scanned and that the
record positions among them are exactly \(R\cap\{i+1,\ldots,n\}\); then the
current threshold when position \(i\) is reached is \(\pi_{\rho_R(i)}\).
Because \(\pi^{-1}\) is a linear extension, the relation between \(i\) and
\(\rho_R(i)\) in \(Q_{\eps,R}\) holds, which by the displayed equivalence
says that \(i\) satisfies the record comparison exactly when \(i\in R\).
Hence \(i\) is declared a record if and only if \(i\in R\), completing the
induction.  Therefore the greedy record set is exactly \(R\).
\end{proof}

\subsection{The record-set Bernstein refinement}

For \(R\subseteq[n]\) with \(n\in R\), set
\[
R_\eps^+(R)=\{i\in R\cap[n-1]:\eps_i=+\},\qquad
R_\eps^-(R)=\{i\in R\cap[n-1]:\eps_i=-\},
\]
and write \(\mathbf y^A=\prod_{i\in A}y_i\).  For a labeled poset \(Q\) on
\([n]\) and \(1\le k\le n\), define
\[
\mathcal L_k(Q)
 =\{\sigma=\sigma_1\cdots\sigma_n\in\mathcal L(Q):\sigma_k=n\}.
\]
For the record poset, define
\[
\mathfrak S_{\eps,R}(k)
 =\{\pi\in S_n:\operatorname{Rec}_\eps(\pi)=R,\ \pi_n=k\}
\]
and
\[
\ell_{\eps,R}(k)=|\mathcal L_k(Q_{\eps,R})|.
\]

Attach a separate variable to the record status of every inspected position.
If \(w=w_1\cdots w_r\), \(\mathbf z=(z_1,\ldots,z_r)\), and
\(\mathbf z'=(z_2,\ldots,z_r)\), define
\(\mathscr K_w^{(p,q)}(x;\mathbf z)\) by
\begin{align*}
\mathscr K_\varnothing^{(p,q)}(x;\varnothing)&=1,\\
\mathscr K_{+v}^{(p,q)}(x;z_1,\mathbf z')
 &=x\mathscr K_v^{(p,q)}(x;\mathbf z')
   +p z_1\int_x^1\mathscr K_v^{(p,q)}(u;\mathbf z')\dd u,\\
\mathscr K_{-v}^{(p,q)}(x;z_1,\mathbf z')
 &=(1-x)\mathscr K_v^{(p,q)}(x;\mathbf z')
   +q z_1\int_0^x\mathscr K_v^{(p,q)}(u;\mathbf z')\dd u.
\end{align*}
Thus \(\mathscr K_w^{(p,q)}(x;1,\ldots,1)=K_w^{(p,q)}(x)\).

For \((g,\theta)\in\mathcal M_w(k)\), put
\(J(g)=S_+(g)\cup S_-(g)\).  When \(n=r+1\) and
\(w=\eps_{n-1}\cdots\eps_1\), define
\[
\mathcal M_w(k;R)
 =\bigl\{(g,\theta)\in\mathcal M_w(k):
          \{n-j:j\in J(g)\}=R\setminus\{n\}\bigr\}.
\]

\begin{theorem}[Record-set Bernstein transfer]
\label{thm:record-set-master-transfer}
Let \(n\ge1\), \(\eps\in\{+,-\}^{n-1}\), and
\(w=\eps_{n-1}\cdots\eps_1\).  For every \(R\subseteq[n]\) with \(n\in R\)
and every \(k\in[n]\), there are explicit bijections
\[
\mathcal M_w(k;R)
 \longleftrightarrow \mathfrak S_{\eps,R}(k)
 \longleftrightarrow \mathcal L_k(Q_{\eps,R}).
\]
In particular,
\[
|\mathcal M_w(k;R)|
 =|\mathfrak S_{\eps,R}(k)|
 =\ell_{\eps,R}(k).
\]
Moreover, with \(r=n-1\),
\begin{align}
&r!\,\mathscr K_w^{(p,q)}
  (x;y_{n-1},y_{n-2},\ldots,y_1) \notag\\
&\quad=
\sum_{k=1}^{n}\ \sum_{\substack{R\subseteq[n]\\n\in R}}
 \ell_{\eps,R}(k)
 p^{|R_\eps^+(R)|}q^{|R_\eps^-(R)|}
 \mathbf y^{R\setminus\{n\}}B_{r,k}(x).
\label{eq:record-set-master-bernstein}
\end{align}
Consequently,
\begin{align}
&y_n n!\int_0^1
 \mathscr K_w^{(p,q)}(x;y_{n-1},\ldots,y_1)\dd x \notag\\
&\quad=
\sum_{\pi\in S_n}
 \mathbf y^{\operatorname{Rec}_\eps(\pi)}
 p^{\rec_\eps^+(\pi)}q^{\rec_\eps^-(\pi)} \notag\\
&\quad=
\sum_{\substack{R\subseteq[n]\\n\in R}}
 \mathbf y^R p^{|R_\eps^+(R)|}q^{|R_\eps^-(R)|}
 |\mathcal L(Q_{\eps,R})|.
\label{eq:record-set-master-integrated}
\end{align}
\end{theorem}

\begin{proof}
Fix \(k\), and write the entries inspected from right to left as
\(\alpha=(a_1,\ldots,a_r)\), where \(a_j=\pi_{n-j}\).  Thus \(\alpha\) is a
permutation of \([n]\setminus\{k\}\).  The bijection in the proof of
\Cref{thm:decorated-signed-transfer} sends \(\alpha\) to its threshold history
\(g\) and stores \(a_j\) as \(\theta_j\) precisely when step \(j\) is not a
record.  A strict step at time \(j\) is equivalent to a record at position
\(n-j\).  Hence that bijection restricts to
\[
\mathcal M_w(k;R)\longleftrightarrow\mathfrak S_{\eps,R}(k).
\]
By \Cref{lem:record-set-linear-extensions}, the map
\(\pi\mapsto\pi^{-1}\) sends the latter set bijectively to linear extensions
of \(Q_{\eps,R}\).  Finally,
\[
\pi_n=k\quad\Longleftrightarrow\quad (\pi^{-1})_k=n,
\]
so its image is exactly \(\mathcal L_k(Q_{\eps,R})\).

For the Bernstein expansion, attach the additional weight \(z_j\) to a
record at scan time \(j\) in the proof of
\Cref{thm:decorated-signed-transfer}.  The same first-step recurrence, with
\(p,q\) replaced by \(pz_j,qz_j\), gives
\[
r!\,\mathscr K_w^{(p,q)}(x;\mathbf z)
 =\sum_{k=1}^{r+1}A_w(k;\mathbf z)B_{r,k}(x).
\]
Here \(A_w(k;\mathbf z)\) is the weighted generating polynomial of scan
orders on \([r+1]\setminus\{k\}\).  Taking \(z_j=y_{n-j}\), the two
bijections above give
\[
A_w(k;y_{n-1},\ldots,y_1)
 =\sum_{\substack{R\subseteq[n]\\n\in R}}
 \ell_{\eps,R}(k)
 p^{|R_\eps^+(R)|}q^{|R_\eps^-(R)|}
 \mathbf y^{R\setminus\{n\}},
\]
which proves \eqref{eq:record-set-master-bernstein}.  Integrating and using
\(\int_0^1B_{r,k}(x)\dd x=1/n\) gives
\eqref{eq:record-set-master-integrated}, since
\(\sum_k\ell_{\eps,R}(k)=|\mathcal L(Q_{\eps,R})|\).
\end{proof}

\subsection{A \texorpdfstring{\(P\)}{P}-partition consequence}
\label{sec:qsym-record-lift}

For a word \(\sigma=\sigma_1\cdots\sigma_n\) with distinct letters, set
\(\operatorname{Des}(\sigma)=\{i:\sigma_i>\sigma_{i+1}\}\).  For
\(S\subseteq[n-1]\), let
\[
F_{S,n}(\mathbf x)
=
\sum_{\substack{1\le a_1\le\cdots\le a_n\\
                 a_i<a_{i+1}\text{ for }i\in S}}
x_{a_1}\cdots x_{a_n}
\]
be Gessel's fundamental quasisymmetric function.  For a labeled poset \(Q\)
on \([n]\), define its pointed enumerators
\[
\Gamma_k(Q)=
\sum_{\sigma\in\mathcal L_k(Q)}
F_{\operatorname{Des}(\sigma),n}(\mathbf x).
\]
Their sum over \(k\) is the usual fundamental \(P\)-partition enumerator
\cite{StanleyPPartitions,Gessel1984}.

For \(k\in[n]\), define
\[
\mathcal R^\bullet_{\eps,k}(\mathbf x;\mathbf y;u,v)
=
\sum_{\substack{\pi\in S_n\\\pi_n=k}}
\mathbf y^{\operatorname{Rec}_\eps(\pi)\setminus\{n\}}
u^{\rec_\eps^+(\pi)}v^{\rec_\eps^-(\pi)}
F_{\operatorname{Des}(\pi^{-1}),n}(\mathbf x).
\]
\begin{corollary}[Fixed-terminal \(P\)-partition lift]
\label[corollary]{cor:qsym-record-lift}
For every \(k\in[n]\),
\[
\mathcal R^\bullet_{\eps,k}
=
\sum_{\substack{R\subseteq[n]\\n\in R}}
\mathbf y^{R\setminus\{n\}}
u^{|R_\eps^+(R)|}v^{|R_\eps^-(R)|}
\Gamma_k(Q_{\eps,R}).
\]
\end{corollary}

\begin{proof}
By \Cref{lem:record-set-linear-extensions}, inversion maps the permutations
with record set \(R\) bijectively to \(\mathcal L(Q_{\eps,R})\).  Moreover,
\(\pi_n=k\) if and only if \((\pi^{-1})_k=n\), so the fixed-terminal fiber
maps to \(\mathcal L_k(Q_{\eps,R})\).  Grouping the defining sum by \(R\)
gives the identity.
\end{proof}

Summing over \(k\) and applying the standard stable principal specialization
\(\operatorname{ps}_q:x_i\mapsto q^{i-1}\), for which
\(\operatorname{ps}_q(F_{S,n})
=q^{\operatorname{comaj}_n(S)}/(q;q)_n\), where
\(\operatorname{comaj}_n(S)=\sum_{i\in S}(n-i)\), gives
\[
(q;q)_n\operatorname{ps}_q\!\left(
y_n\sum_{k=1}^n\mathcal R^\bullet_{\eps,k}\right)=
\sum_{\pi\in S_n}
q^{\operatorname{comaj}_n(\operatorname{Des}(\pi^{-1}))}
\mathbf y^{\operatorname{Rec}_\eps(\pi)}
u^{\rec_\eps^+(\pi)}v^{\rec_\eps^-(\pi)}.
\]
Here \((q;q)_n=\prod_{i=1}^n(1-q^i)\).

\section{Record caterpillars and terminal values}
\label{sec:record-tree-spectrum}

Atkinson's algorithm for posets whose cover graph is a tree proceeds through the
position spectrum of a distinguished element and runs in quadratic time
\cite{AtkinsonTree}.  Record posets have caterpillar cover graphs, so its
recursion has the following explicit form in record-gap coordinates.  By
\Cref{thm:record-set-master-transfer}, this distinguished-element spectrum is
the terminal-value distribution in a fixed record-set fiber.

\begin{lemma}[Record-caterpillar structure]
\label[lemma]{lem:record-caterpillar}
Write \(R=\{r_1<\cdots<r_s=n\}\) and set \(r_0=0\).  The undirected Hasse
diagram of \(Q_{\eps,R}\) is a caterpillar with spine
\[
r_1-r_2-\cdots-r_s.
\]
For \(r_{j-1}<i<r_j\), the vertex \(i\) is a leaf adjacent to \(r_j\), and
\[
i<_{Q_{\eps,R}}r_j\quad\Longleftrightarrow\quad\eps_i=+.
\]
For \(2\le j\le s\), the orientation of the \(j\)-th spine edge is
\[
r_{j-1}<_{Q_{\eps,R}}r_j
 \quad\Longleftrightarrow\quad \eps_{r_{j-1}}=-.
\]
\end{lemma}

\begin{proof}
If \(i\notin R\), then \(\rho_R(i)=r_j\) for the unique \(j\) satisfying
\(r_{j-1}<i<r_j\).  No defining edge can have such an \(i\) as its right
endpoint, so \(i\) is a leaf.  If \(i=r_{j-1}\in R\), then
\(\rho_R(i)=r_j\), producing the spine edge.  The two orientation statements
are immediate from \Cref{def:record-tree}.  The defining undirected graph is a
tree, so none of its edges can be made redundant by transitive closure;
therefore it is the undirected Hasse diagram.
\end{proof}

For \(1\le j\le s\), let \(Q_j\) be the subposet of \(Q_{\eps,R}\) induced by
\([r_j]\), and define the root spectrum
\[
h_j(k)=\#\{\sigma\in\mathcal L(Q_j):\sigma_k=r_j\}
\qquad(1\le k\le r_j).
\]
Also put
\[
a_j=\#\{i:r_{j-1}<i<r_j,\ \eps_i=+\},\qquad
b_j=\#\{i:r_{j-1}<i<r_j,\ \eps_i=-\}.
\]
Thus \(a_j+b_j=r_j-r_{j-1}-1\).

\begin{proposition}[Record-gap specialization of Atkinson's recursion]
\label{thm:record-tree-spectrum}
The initial spectrum is concentrated at one position:
\[
h_1(k)=
\begin{cases}
a_1!b_1!,&k=a_1+1,\\
0,&k\ne a_1+1.
\end{cases}
\]
For \(2\le j\le s\) and \(0\le d\le r_{j-1}\), define
\[
C_j(d)=
\begin{cases}
\displaystyle\sum_{t=1}^{d}h_{j-1}(t),
 &\eps_{r_{j-1}}=-,\\[3mm]
\displaystyle\sum_{t=d+1}^{r_{j-1}}h_{j-1}(t),
 &\eps_{r_{j-1}}=+.
\end{cases}
\]
Then
\begin{equation}
\label{eq:record-tree-spectrum-recursion}
h_j(d+a_j+1)
 =a_j!b_j!
  \binom{d+a_j}{a_j}
  \binom{r_{j-1}-d+b_j}{b_j}
  C_j(d).
\end{equation}
All other entries of \(h_j\) are zero.  In particular,
\[
\ell_{\eps,R}(k)=h_s(k),
\qquad
\#\{\pi\in S_n:\operatorname{Rec}_\eps(\pi)=R\}
 =\sum_{k=1}^{n}h_s(k).
\]
After factorials and binomial coefficients have been precomputed, the entire
vector \((h_s(1),\ldots,h_s(n))\) is obtained in \(O(n^2)\) arithmetic
operations.
\end{proposition}

\begin{proof}
For \(j=1\), the \(a_1\) lower leaves must occur before \(r_1\), the \(b_1\)
upper leaves must occur after \(r_1\), and the leaves within either group may
be ordered arbitrarily.  This gives the stated initial spectrum.

Fix \(j\ge2\) and a linear extension \(\tau\) of \(Q_{j-1}\).  When \(r_j\)
is inserted, let \(d\) be the number of entries of \(\tau\) placed before
\(r_j\).  If \(r_{j-1}<_{Q_{\eps,R}}r_j\), the entry \(r_{j-1}\) must belong
to this prefix, so its position in \(\tau\) is at most \(d\).  If
\(r_j<_{Q_{\eps,R}}r_{j-1}\), it must belong to the suffix, so its position is
larger than \(d\).  By \Cref{lem:record-caterpillar}, the number of admissible
choices of \(\tau\) is therefore exactly \(C_j(d)\).

For a fixed admissible \(\tau\), interleave its ordered prefix of length \(d\)
with the \(a_j\) distinct lower leaves.  This can be done in
\[
a_j!\binom{d+a_j}{a_j}
\]
ways.  Independently, interleave its ordered suffix of length
\(r_{j-1}-d\) with the \(b_j\) distinct upper leaves, in
\[
b_j!\binom{r_{j-1}-d+b_j}{b_j}
\]
ways.  The root \(r_j\) lies between these two words and hence occupies
position \(d+a_j+1\).  This construction is reversible by restricting a
linear extension of \(Q_j\) to \([r_{j-1}]\) and to the two leaf sets, so it
counts every extension exactly once and proves
\eqref{eq:record-tree-spectrum-recursion}.  The final identities follow from
\Cref{thm:record-set-master-transfer}.  Prefix or suffix sums compute every
array \(C_j\) in \(O(r_{j-1})\) operations, and summing over \(j\) gives the
stated quadratic bound.
\end{proof}

\begin{remark}[Terminal-value range]\label{cor:terminal-value-interval}
For any finite poset \(Q\) and \(x\in Q\), the possible positions of \(x\)
among the linear extensions form the interval
\[
\bigl[|Q_{<x}|+1,\ |Q|-|Q_{>x}|\bigr]\cap\mathbb Z.
\]
For the record caterpillar this gives integers \(L_j\le U_j\) such that
\[
h_j(k)>0\quad\Longleftrightarrow\quad L_j\le k\le U_j.
\]
In gap coordinates, \(L_1=U_1=a_1+1\) and
\[
(L_j,U_j)=
\begin{cases}
(a_j+L_{j-1}+1,\ r_j-b_j),&\eps_{r_{j-1}}=-,\\
(a_j+1,\ a_j+U_{j-1}),&\eps_{r_{j-1}}=+.
\end{cases}
\]
Consequently,
\[
\{\pi_n:\operatorname{Rec}_\eps(\pi)=R\}=[L_s,U_s]\cap\mathbb Z.
\]
\end{remark}

\begin{remark}[Gap-data invariance]\label{cor:block-count-invariance}
Fix \(R=\{r_1<\cdots<r_s=n\}\).  Suppose that
\(\eps,\eps'\in\{+,-\}^{n-1}\) have the same signs at
\(r_1,\ldots,r_{s-1}\) and the same number of \(+\)-signs in every open gap
\((r_{j-1},r_j)\).  Then, for every \(k\in[n]\),
\[
\#\{\pi:\operatorname{Rec}_\eps(\pi)=R,\ \pi_n=k\}
=
\#\{\pi:\operatorname{Rec}_{\eps'}(\pi)=R,\ \pi_n=k\}.
\]
Indeed, the hypotheses give an isomorphism of rooted record caterpillars:
match the spine vertices and, in each gap, match \(+\)-leaves to \(+\)-leaves
and \(-\)-leaves to \(-\)-leaves.  The isomorphism fixes the distinguished
vertex \(n\), so the spectra agree.
\end{remark}

\begin{example}\label[example]{ex:record-tree-spectrum}
Let \(\eps=(+,-,+)\) and \(R=\{1,2,4\}\).  The record poset has cover
relations
\[
2<1,\qquad 2<4,\qquad 3<4.
\]
The recursion gives
\[
\begin{aligned}
(h_1(1))&=(1),&
(h_2(1),h_2(2))&=(1,0),\\
(h_3(1),\ldots,h_3(4))&=(0,0,2,3).
\end{aligned}
\]
Indeed, its five linear extensions are
\[
2134,\quad2314,\quad2341,\quad3214,\quad3241.
\]
The corresponding inverse permutations are
\[
2134,\quad3124,\quad4123,\quad3214,\quad4213.
\]
Two have terminal value \(3\), three have terminal value \(4\), and all have
record set \(R\).  Thus the contribution of this fixed record set to
\eqref{eq:record-set-master-bernstein} is
\[
pq\,y_1y_2\bigl(2B_{3,3}(x)+3B_{3,4}(x)\bigr).
\]
\end{example}

\section{Concluding remarks}

The path construction is refined above by endpoint, record-set, direction,
and terminal-value data, while the cyclic construction replaces the
distinguished endpoint by a canonical root.  It remains open whether the
record-set Bernstein transfer and its \(P\)-partition consequence admit
cyclic analogues.

\section*{Acknowledgements}
The author thanks Neil J.~Y. Fan for drawing attention to Problem~5.3 of
Ferroni--Morales--Panova and for his encouragement, and Quanyu Tang for helpful
discussions and suggestions.

\end{document}